\documentclass{article}
\usepackage{amsmath}
\usepackage{bm}
\usepackage{graphicx}
\usepackage{amssymb}
\usepackage{ulem}
\usepackage[dvipsnames,svgnames]{xcolor}
\usepackage{caption}
\everymath{\displaystyle}
\DeclareMathOperator{\sgn}{sgn}

\usepackage{natbib}
\usepackage{svg}
\usepackage{amsthm}
\usepackage{geometry}
\usepackage{float}
\usepackage{setspace}
\usepackage{indentfirst}
\usepackage[T1]{fontenc}
\usepackage[utf8]{inputenc}
\usepackage{authblk}
\usepackage{fancyhdr}
\usepackage{afterpage}
\usepackage{hyperref}
\usepackage{textpos}
\usepackage{atbegshi}
\usepackage{marvosym}
\graphicspath{{figures/}}
\theoremstyle{plain}
\newtheorem{theorem}{Theorem}[section]
\newtheorem{lemma}{Lemma}[section]
\newtheorem{corollary}{Corollary}[section]
\newtheorem{proposition}{Proposition}[section]
\theoremstyle{definition}
\newtheorem{definition}{Definition}[section]
\newtheorem{example}{Example}[section]
\theoremstyle{nonumber}

\begin{document}

	\normalem
	\setstretch{1.3}
	
	\title{Coloring-allowed Invariants and the 4-phases Functions of Knotoids}
	
	\author[1]{Haocong Chen\thanks{412712chc1728@mail.dlut.edu.cn}}
	\author[1]{Jiacheng An\thanks{anan0312@163.com}}
	\author[1]{Fengling Li\thanks{{fenglingli@dlut.edu.cn}}}
	\affil[1]{School of Mathematical Sciences, Dalian University of Technology, Dalian 116024, P. R. China}
	\date{}
	\maketitle

	\begin{abstract}
	\textbf{Abstract.} 
	In recent years, numerous polynomial invariants of knotoids have been constructed, some of which are defined with the signs of the crossings. In this paper, the coloring-allowed invariants of planar knotoids, which is a class of planar knotoid invariants defined with the coloring number are introduced. We demonstrate several basic properties of the coloring number and some examples of coloring-allowed invariants, among which the 4-phases functions are discussed in detail, including their invariance and properties.\\
	
	\end{abstract}
	
	$Keywords:$ Planar knotoid; Gauss diagram; coloring number; coloring-allowed invariant; 4-phases function

	\section{Introduction}

	Knotoids can be viewed as knots which are not "closed". The concept of knotoid was first proposed by Turaev in \cite{1} in 2012. A knotoid diagram is a generic immersion of an oriented unit interval $[0,1]$ into an oriented surface, with over-under information assigned at double points. Two knotoid diagrams on a surface are equivalent if and only if they can be transformed into each other by a finite number of Reidemeister moves and isotopies on this surface. The corresponding equivalence classes are called knotoids in the surface. See \cite{1} for a detailed introduction.
	
	The theory of knotoids can be regarded as a generalization of classical knot theory, and many invariants of knots can be extended to knotoids. Various studies on invariants of knotoids have been presented so far. Gügümcü and Kauffman generalized knotoids to virtual knotoids by introducing virtual crossings and constructed several new invariants in \cite{7}. In \cite{17}, Gügümcü and Nelson considered biquandle colorings for knotoids in $\mathbb{R}^2$ or $S^2$, and they constructed several coloring invariants for knotoids derived as enhancements of the biquandle counting invariant. In \cite{22}, Gügümcü and Kauffman constructed quantum invariants for planar Morse knotoid diagrams. A generalization of quantum knot invariants to a knotoid setting was given by Moltmaker in \cite{21}. In \cite{18}, Moltmaker and van der Veen showed how invariants of knotoids generally give rise to well-behaved measures of how much an open curve is knotted, and improved the classification of planar knotoids with up to five crossings with these invariants. In \cite{19}, Elhamdadi, Moltmaker and Saito defined several new planar knotoid invariants based on quandles which are able to detect planarity. Miyazawa defined a multi-variable polynomial invariant for knotoids and linkoids, which is an enhancement of the bracket polynomial for knotoids introduced by Turaev in \cite{20}.

	In recent years, numerous polynomial invariants of classical or virtual knots and knotoids have been constructed, among which some are defined via Gauss diagrams or the weight of the crossings. For instance, Kauffman defined the affine index polynomial of a knotoid diagram by
	\begin{center}
		$P_{D}(t)=\sum_{c\in C(D)}^{} \sgn(c)(t^{d(c)}-1 )$,
	\end{center}
	where $d(c)$ is the degree of a chord in the Gauss diagram of $D$, which is referred to as the weight of a crossing in \cite{12}.
	In \cite{15}, Kim, Im and Lee introduced the index polynomial of knotoids and the $n$th polynomial invariant of virtual knotoids by means of Gauss diagrams.
	The index polynomial $F_{D}(t)$ of a knotoid diagram $D$ is defined as follows:
	\begin{center}
		$F_{D}(t)=\sum_{c\in C(D)}^{} \sgn(c)(t^{i(c)}-1 )$.
	\end{center}
	Feng and Li gave a new two-variable invariant of knotoids via Gauss diagrams in \cite{8}, which is defined as:
	\begin{center}
		$F_{D}(u,v)=\sum_{c\in C(D)}^{} \sgn(c)(u^{\widetilde{h}_c(v)} -1)$.
	\end{center}
	Feng, Li and Vesnin proposed a new invariant with 3 variables of knotoids in \cite{9}, which is defined by
	\begin{center}
		$H_{D}(x,y,v)=\sum_{c\in C(D)}^{}\sum_{n\in\mathbb{N}} \sgn(c)(x^{g_c^n(v)}-1 )y^n$.\\
	\end{center}

	The sign of a crossing $c$ $\sgn(c)$ appears in each of the invariants mentioned above. The effect of an $\Omega_1$ move on these invariants has to be taken into consideration when one constructs them. To avoid this, we will try to replace it by the so-called coloring number $u(c)$, which can be considered as a modification to the weight of a crossing. The notion of the weight of a crossing has been introduced in \cite{7}, \cite{10} and \cite{12}, and the coloring number has been introduced in \cite{3} and \cite{4}, where it is still called the "weight" of a crossing. Kauffman has proved in \cite{12} that the value of the weight at each crossing in an arbitrary classical knot diagram is 0, and so is the coloring number. Hence it is futile to construct invariants of classical knots with the coloring number as a factor of the coefficient of each term, since one can only obtain trivial invariants by this approach. Therefore we consider invariants of knotoids defined with the coloring number, namely the so-called coloring-allowed invariants of knotoids. As an example of coloring-allowed invariants, the 4-phases functions will be discussed in detail.

	In this paper, some basic notions are given in Section 2. In Section 3, we illustrate the definition and properties of coloring numbers, after which the definition of coloring-allowed invariants and some examples will be given. Four coloring-allowed invariants which look quite similar, namely the 4-phases functions, are defined in Section 4, and their invariance is verified in this section. Discussions related to the properties of the 4-phases invariants are presented in the last part of Section 4.

	\section{Preliminaries}

	At the beginning, let us give some basic notions and notations. Let $\Sigma$ be an oriented surface. A knotoid diagram $D$ in $\Sigma$ is a generic immersion of the unit interval $[0,1]$ into the interior of $\Sigma$. In particular, knotoid diagrams on $\Sigma=\mathbb{R}^2$ are referred to as planar knotoid diagrams. In this paper, a knotoid diagram is always assumed to be oriented, where the orientation is induced by the canonical orientation of the unit interval $[0,1]$ from $0$ to $1$. The images of $0$ and $1$, which are distinct from the crossings, are called the leg and the head of the knotoid diagram respectively.

	Geometrically, a planar knotoid diagram is a curve on the Euclidean plane that transversely intersects itself at some isolated crossings which are endowed with over/undercrossing information. To be precise, a crossing $c$ of a knotoid diagram $D$ is a point whose preimage under the immersion consists of exactly 2 elements. The set consisting of all crossings in $D$ is denoted by $C(D)$, namely
	\begin{center}
		$C(D)=\left\lbrace c\in D\mid f^{-1}(\left\lbrace c\right\rbrace  )=\left\lbrace a,b\right\rbrace ,0<a<b<1\right\rbrace $,
	\end{center}
	where $f$ is the immersion corresponding to the knotoid diagram $D$. The sign of a crossing $c$ in a knotoid diagram is defined by the way shown in Figure \ref{1}, and it is usually denoted by $\sgn (c)$. Every connected component of $D\setminus C(D)$ is called an open arc of $D$, and the closure of an open arc in $\mathbb{R}^2$ is called an arc of $D$. The set consisting of all arcs in $D$ is denoted by $A(D)$.

	\begin{figure}[ht]
		\centering
		\includegraphics[width=0.36\textwidth]{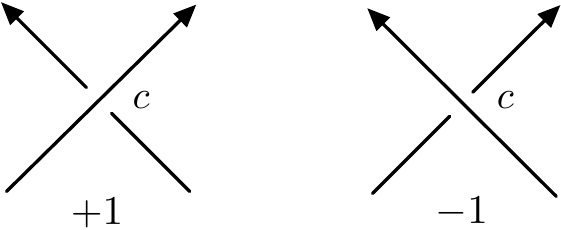}	
		\caption{The sign of a crossing}
		\label{1}
	\end{figure}

	Analogous to classical knots, the Reidemeister moves can be defined on knotoid diagrams as well, as shown in Figure \ref{2}. When a knotoid diagram is endowed with an orientation, the corresponding moves are called oriented Reidemeister moves. For two oriented planar knotoid diagrams $D_1$ and $D_2$, they are said to be equivalent if one can be transformed into the other via a sequence of oriented Reidemeister moves and local planar isotopies. The corresponding equivalence classes are called planar knotoids.

	\begin{figure}[ht]
		\centering
		\includegraphics[width=0.96\textwidth]{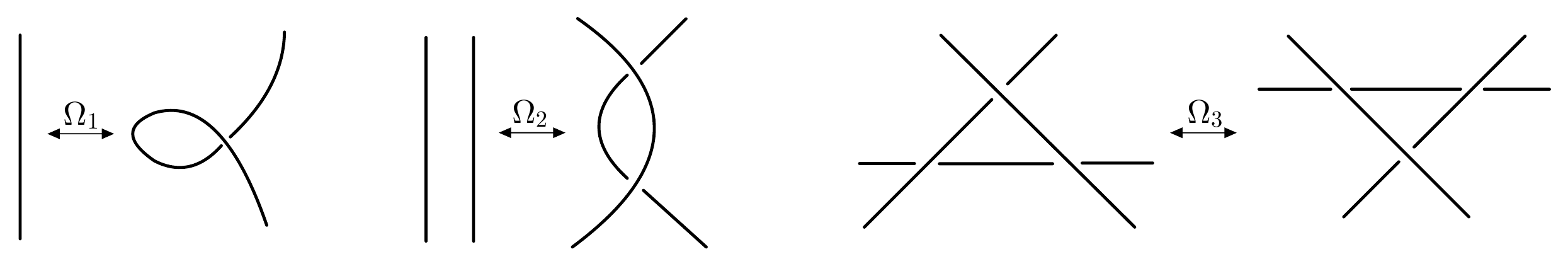}
		\caption{Three Reidemeister moves}
		\label{2}
	\end{figure}

	The knotoid diagram represented by an embedding $[0,1]\hookrightarrow \Sigma$ is said to be a trivial knotoid diagram. In Figure \ref{3}, the moves $\Omega_+$ and $\Omega_-$ are referred to as forbidden moves, since if the $\Omega_+$ and $\Omega_-$ moves were permitted, then any knotoid diagram could be easily reduced to a trivial knotoid diagram.

	\begin{figure}[ht]
		\centering
		\includegraphics[width=0.48\textwidth]{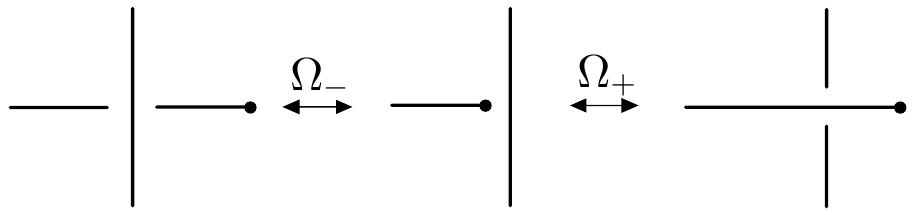}
		\caption{Forbidden moves}
		\label{3}
	\end{figure}

	The Gauss diagram $G(D)$ of a knotoid diagram $D$ is a counterclockwise arc together with several chords, each of which connects the preimages of a crossing in $D$. For each crossing $c$ in a knotoid diagram, there is a unique corresponding chord in its Gauss diagram, hence there is a natural bijection between $C(D)$ and the set consisting of all chords in $G(D)$. For the sake of simplicity, the chord in $G(D)$ corresponding to $c\in C(D)$ is also denoted by $c$. A chord is oriented from the preimage of the overpassing arc to the preimage of the underpassing arc of the corresponding crossing. Moreover, $\sgn(c)$ and $-\sgn(c)$ are marked at the starting point and ending point of the chord $c$ respectively, as shown in Figure \ref{4}. Sometimes $-\sgn(c)$ marked at the ending point of a chord can be omitted.

	\begin{figure}[ht]
		\centering
		\includegraphics[width=0.16\textwidth]{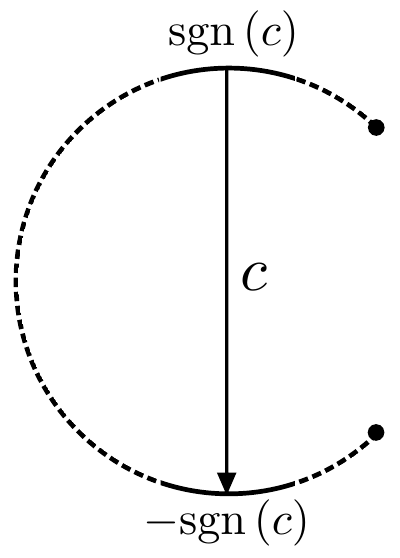}
		\caption{Gauss diagram}
		\label{4}
	\end{figure}

	Now let us turn to some notions which are related to the index function $g_c(v)$. For a knotoid diagram $D$ and a fixed chord $c\in C(D)$ in the corresponding Gauss diagram $G(D)$, denote the set of chords that pass through $c$ and point to the left hand side of $c$ by $l(c)=\left\lbrace l_1(c),l_2(c),\cdots,l_m(c)\right\rbrace$, and the set of chords that pass through $c$ and point to the right hand side of $c$ by $r(c)=\left\lbrace r_1(c),r_2(c),\cdots,r_n(c)\right\rbrace$, as shown in Figure \ref{5}(1).

	\begin{figure}[ht]
		\centering
		\includegraphics[width=0.48\textwidth]{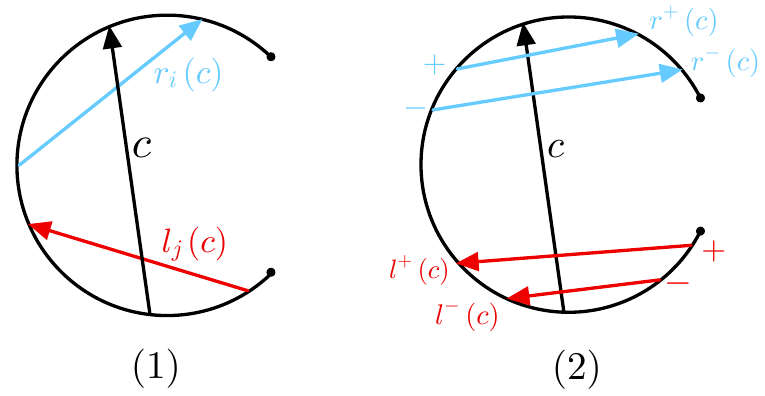}
		\caption{$l(c)$ and $r(c)$}
		\label{5}
	\end{figure}

	For a knotoid diagram $D$ and a fixed chord $c\in C(D)$, let
	\begin{center}
		$
		\begin{aligned}
			r^+(c)&=\left\lbrace c'\in r(c)\mid\sgn(c')=1\right\rbrace; \\
			r^-(c)&=\left\lbrace c'\in r(c)\mid\sgn(c')=-1\right\rbrace; \\
			l^+(c)&=\left\lbrace c'\in l(c)\mid\sgn(c')=1\right\rbrace; \\
			l^-(c)&=\left\lbrace c'\in l(c)\mid\sgn(c')=-1\right\rbrace .
		\end{aligned}
		$
	\end{center}
	The degree of $c$, which has been mentioned in \cite{13}, \cite{5}, \cite{6} and \cite{9}, is defined by
	\begin{center}
		$d(c)=|r^+(c)|-|r^-(c)|-|l^+(c)|+|l^-(c)|$,
	\end{center}

	\noindent where $|X|$ denotes the cardinality of a set $X$. In particular, if a chord $c$ is isolated, which is to say that no other chord has non-empty intersection with $c$, then $d(c)=0$.

	By the definition of $d(c)$, it is the sum of signs of chord endpoints on the left hand side of $c$ except the signs of endpoints of $c$. For each fixed chord $c$ in $G(D)$, we define a function $\phi_c:\mathbb{Z}\to \mathbb{Z}_{|d(c)|}$ as following:
	\begin{center}
		$\phi_c(k)\equiv k \pmod {|d(c)|} $,
	\end{center}
	where the elements of $\mathbb{Z}_{|d(c)|}$ are considered to be the representative elements of the residue classes here, namely  $\mathbb{Z}_{|d(c)|}=\left\lbrace 0,1,\cdots,|d(c)|-1\right\rbrace $. In this way, it is obvious that $\phi_c(k)$ is always uniquely determined by $k$, hence there is no ambiguity in the definition above. In particular, if $d(c)=0$, the image set of $\phi_c$ happens to be an empty set $\varnothing$. In this case, $\phi_c$ is defined to be an identity map on $\mathbb{Z}$. A salient feature of this function is that it is a homomorphism.

	With the function $\phi_c$, we can define the index function of a chord $c$ in $G(D)$. The index function $g_c(v)$ is defined by
	\begin{center}
		$g_c(v)=\sum_{e\in r(c)}\sgn(e)v^{\phi_c(d(e))}-\sum_{e\in l(c)}\sgn(e)v^{\phi_c(-d(e))}$.\\
	\end{center}
	By the definition of $\phi_c$, it is trivial that $v^{|d(c)|}=v^0=1$, hence
	$g_c(v)$ takes values in $\mathbb{Z}[v,v^{-1}]/(v^{|d(c)|}-1)$. In particular, when $v=1$ or $|d(c)|=1$, 
	\begin{center}
		$
		\begin{aligned}
			g_c(v)=\sum_{e\in r(c)}\sgn(e)-\sum_{e\in l(c)}\sgn(e)
			=\sum_{e\in r(c)}\sgn(e)+\sum_{e\in l(c)}\left( -\sgn(e)\right) 
		\end{aligned}
		$.\\
	\end{center}
	By the definition of a Gauss diagram,  $\sgn(e)$ coincides with the sign of the starting point of $e$, and $-\sgn(e)$ is the sign of the ending point of $e$.   Consequently, 
	\begin{center}
		$
		\begin{aligned}
			g_c(1)
			=\sum_{e\in r(c)}\sgn(e)+\sum_{e\in l(c)}\left( -\sgn(e)\right) =d(c),
		\end{aligned}
		$\\
	\end{center}
	since those chords whose endpoints are both on the left hand side of $c$ contribute $1+(-1)=0$ to $d(c)$.

	Another useful fact about the function $\phi_c$ is that it is an odd function when it is the exponent of $v$, namely $v^{\phi_c(-k)}=v^{-\phi_c(k)}$. Indeed, since $\phi_c$ is a homomorphism which has been shown above, we have $\phi_c(-k)+\phi_c(k)=\phi_c((-k)+k)=\phi_c(0)=0$. It follows that $v^{\phi_c(-k)}\cdot v^{\phi_c(k)}=v^{\phi_c(-k)+ \phi_c(k)}=v^0=1$.
	Therefore $v^{\phi_c(-k)}=v^{-\phi_c(k)}$.\\

	For any chord $c$ of $G(D)$, the counter-clockwise oriented arc $C$ on the boundary of $G(D)$ is divided into two sides by the chord $c$, one of which containing the head and the tail is called the lead side of $c$. Denote the set of chords that pass through $c$ and point to the lead side of $c$ by $\lambda(c)=\left\lbrace \lambda_1(c),\lambda_2(c),\cdots,\lambda_m(c)\right\rbrace$, and the set of chords that pass through $c$ and point from the lead side of $c$ to the other side by $\rho(c)=\left\lbrace \rho_1(c),\rho_2(c),\cdots,\rho_n(c)\right\rbrace$, as shown in Figure \ref{6}.

	\begin{figure}[ht]
		\centering
		\includegraphics[width=0.18\textwidth]{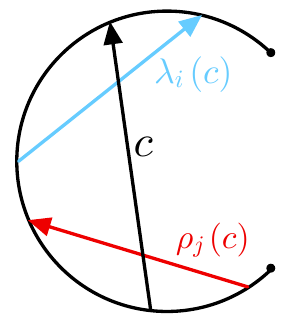}
		\caption{$\lambda(c)$ and $\rho(c)$}
		\label{6}
	\end{figure}

	Next we slightly modify the original definition of $g_c(v)$ to obtain another index function $h_c(v)$. Replacing $r(c)$ and $l(c)$ in the definition of $g_c(v)$ by $\rho(c)$ and $\lambda(c)$ respectively, the index function $h_c(v)$ is defined by
	\begin{center}
		$h_c(v)=\sum_{e\in \rho(c)}\sgn(e)v^{\phi_c(d(e))}-\sum_{e\in\lambda(c)}\sgn(e)v^{\phi_c(-d(e))}$.\\
	\end{center}
	Analogous to $g_c(v)$,
	$h_c(v)$ takes values in $\mathbb{Z}[v,v^{-1}]/(v^{|d(c)|}-1)$ as well. If the left hand side of $c$ is exactly the lead side of $c$, then $l(c)=\lambda(c)$ and $r(c)=\rho(c)$, which indicates that $g_c(v)$ coincides with $h_c(v)$. Otherwise, if the right hand side of $c$ is the lead side of $c$, then $r(c)=\lambda(c)$ and $l(c)=\rho(c)$, which gives that 
	\begin{center}
		$
		\begin{aligned}
			h_c(v)&=\sum_{e\in \rho(c)}\sgn(e)v^{\phi_c(d(e))}-\sum_{e\in\lambda(c)}\sgn(e)v^{\phi_c(-d(e))}\\
			&=\sum_{e\in l(c)}\sgn(e)v^{\phi_c(d(e))}-\sum_{e\in r(c)}\sgn(e)v^{\phi_c(-d(e))}\\
			&=\sum_{e\in l(c)}\sgn(e)(v^{-1})^{\phi_c(-d(e))}-\sum_{e\in r(c)}\sgn(e)(v^{-1})^{\phi_c(d(e))}\\
			&=-\left( \sum_{e\in r(c)}\sgn(e)(v^{-1})^{\phi_c(d(e))}-\sum_{e\in l(c)}\sgn(e)(v^{-1})^{\phi_c(-d(e))}\right) \\
			&=-g_c(v^{-1}).
		\end{aligned}
		$
	\end{center}

	Next we will introduce a new notion called the index of a chord. 	The intersection index of $c$, which is denoted by $i(c)$, is defined to be the sum of signs of chord endpoints on the lead side of $c$ except the signs of endpoints of $c$. If the left hand side of $c$ is exactly the lead side of $c$, then obviously $d(c)=i(c)$. Otherwise, if the right hand side of $c$ is the lead side of $c$, then $d(c)=-i(c)$, since for a chord $c'$ in $G(D)$ with $c'\cap c\neq \varnothing$, one of its endpoints lies on the left hand side of $c$ and the other one lies on the right hand side, hence
	\begin{center}
		$i(c)+d(c)=\sum\limits_{c'}\left( \sgn(c')+(-\sgn(c'))\right) =0$,
	\end{center}
	where the sum is taken over the set of chords in $G(D)$ that satisfy $c'\cap c\neq \varnothing$. In a word, $|d(c)|=|i(c)|$.

	Two functions which look quite similar to $g_c(v)$ and $h_c(v)$ will be given below. The index function $\widetilde{g}_c(v)$ is defined by
	\begin{center}
		$\widetilde{g}_c(v)=\sum_{e\in r(c)}\sgn(e)v^{\phi_c(i(e))}-\sum_{e\in l(c)}\sgn(e)v^{\phi_c(-i(e))}$,
	\end{center}	
	and the index function $\widetilde{h}_c(v)$ is defined by
	\begin{center}
		$\widetilde{h}_c(v)=\sum_{e\in \rho(c)}\sgn(e)v^{\phi_c(i(e))}-\sum_{e\in \lambda(c)}\sgn(e)v^{\phi_c(-i(e))}$.\\
	\end{center}
	If the left hand side of $c$ is exactly the lead side of $c$, then $l(c)=\lambda(c)$ and $r(c)=\rho(c)$, which indicates that $\widetilde{g}_c(v)$ coincides with $\widetilde{h}_c(v)$. Otherwise, if the right hand side of $c$ is the lead side of $c$, then a proof analogous to the one we have given when discussing the relationship between $g_c(v)$ and $h_c(v)$ earlier yields $\widetilde{h}_c(v)=-\widetilde{g}_c(v^{-1})$.\\

	To avoid ambiguity, when there are more than 1 knotoid diagrams in a proposition, let an object with a subscript $D$ denote the corresponding notion in $D$ to distinguish the same notion which appears in different knotoid diagrams. For instance, in a knotoid diagram $D$, the function $g_c(v)$ is denoted by $g_{c,D}(v)$, and $d(c)$ is denoted by $d_D(c)$, where $c\in C(D)$. For the sake of simplicity, if there is a natural correspondence between a subset of $C(D_1)$ and a subset of $C(D_2)$, where $D_1$ and $D_2$ are 2 planar knotoid diagrams, then the corresponding crossings will be denoted by the same notation. For instance, for a knotoid diagram $D$ and its inverse $-D$, the function $g_c(v)$ is denoted by $g_{c,D}(v)$ and $g_{c,-D}(v)$ respectively.\\

	\section{Coloring-allowed invariants}

	First we will introduce some notions and illustrate several properties about the coloring number $u(c)$. We label each arc $\gamma$ in the knotoid diagram with an integer $\varphi(\gamma)$ so that an arc $\alpha$ that meets a crossing point and crosses to the left (with respect to the orientation of the arc that $\alpha$ meets) increases the label by 1 ; while an arc $\beta$ that meets a crossing point and crosses to the right (with respect to the orientation of the arc that $\beta$ meets) decreases the label by 1, as shown in Figure \ref{7}. In this way, we have a map $\varphi:A(D)\to \mathbb{Z}$, which is called the Cheng coloring of the knotoid diagram. One may refer to \cite{10}, \cite{7}, \cite{3}, \cite{4} and \cite{12} for more details.

	\begin{figure}[ht]
		\centering
		\includegraphics[width=0.42\textwidth]{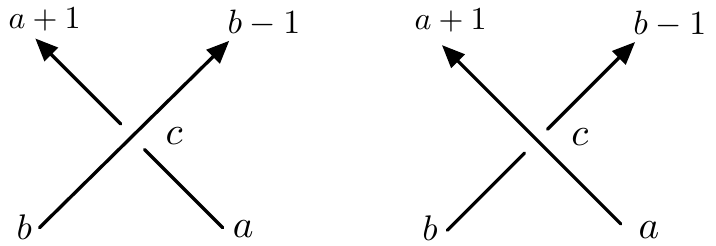}
		\caption{The coloring number}
		\label{7}
	\end{figure}

	We can start by choosing the arc $\gamma_0$ which contains the leg of the knotoid diagram to own an arbitrary integer label $\varphi(\gamma_0)$ , and then proceed along the diagram to label all the arcs via the rule introduced above to obtain a Cheng coloring of the diagram. Two different Cheng colorings of a knotoid diagram only differ by a constant which is an integer. 
	
	With a Cheng coloring of a knotoid diagram, we may define the coloring number of each crossing in it. For a crossing $c$ of a knotoid diagram $D$, take a neighbourhood of $c$ and rotate it such that the directions of the four arcs related to $c$ are obliquely upward, as shown in Figure \ref{7}. We may choose a Cheng coloring $\varphi$ of $D$ such that the labels of the four arcs around $c$ are the ones shown in Figure \ref{7}. Then the coloring number of $c$ is defined by 
	\begin{center}
		$u(c)=a-b$.
	\end{center}
	Due to the fact that two different Cheng colorings of a knotoid diagram only differ by a constant which is an integer, the coloring number $u(c)$ does not depend on the choice of the Cheng coloring, hence it is well-defined. Moreover, the coloring number is independent on the crossing type of the crossing, as shown in Figure \ref{7}.

	The coloring number is an effective tool to construct invariants of knotoids, since one of its advantages is that it behaves well at the crossings related to the $\Omega$-moves. Some properties of the coloring number $u(c)$ will be presented below.

	\begin{lemma}\label{Lemma 3.1.}
		(1) $u(c_0)=0$, where $c_0$ is the new crossing point generated by an $\Omega_1a$ or $\Omega_1 b$ move, and $u(c)$ preserves during this move for an arbitrary crossing $c$ which is not involved in it.\\
		(2) $u(c_2)=-u(c_1)$, where $c_1$ and $c_2$ are the crossings generated by an $\Omega_2 a$ move, and $u(c)$ preserves during an $\Omega_2 a$ move for an arbitrary crossing $c$ which is not involved in it.\\
		(3) $u_D(c_i)=u_{D'}(c_i)$, where $D$ and $D'$ are two oriented knotoid diagrams related by an $\Omega_3 a$ move, $i=1,2,3$, and $u(c)$ preserves during an $\Omega_3 a$ move for an arbitrary crossing $c$ other than $c_1$, $c_2$ and $c_3$.
	\end{lemma}

	\begin{figure}[ht]
		\centering
		\includegraphics[width=\textwidth]{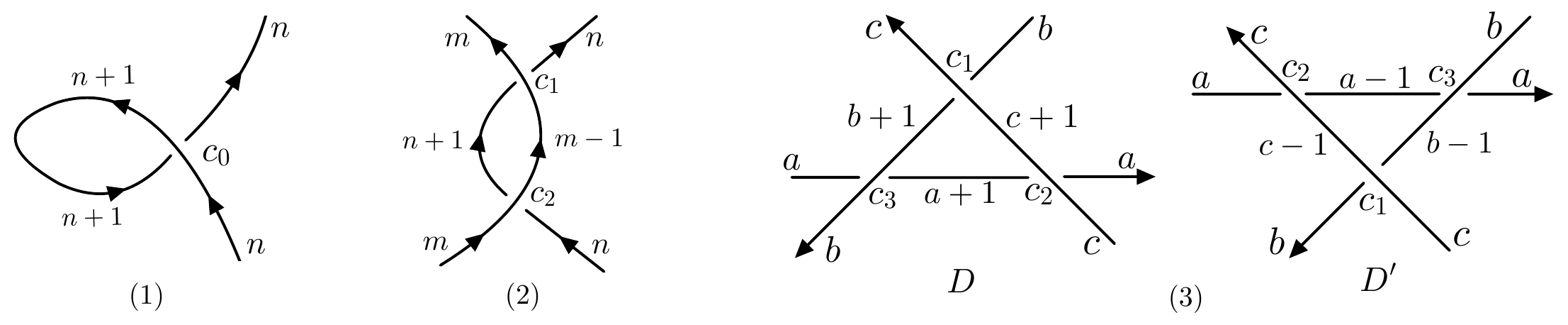}
		\caption{Labels around the crossings relevant to $\Omega$-moves}
		\label{8}
	\end{figure}

	\begin{lemma}\label{Lemma 3.2.} Let $D$ be a knotoid diagram, $-D$ be the inverse of $D$ and
		$D^*$ be the mirror image of $D$. For an arbitrary crossing $c$ in $D$, we have $u_{-D}(c)=-u_D(c)$ and 
		$u_{D^*}(c)=u_D(c)$.
	\end{lemma}

	\begin{figure}[ht]
		\centering
		\includegraphics[width=0.56\textwidth]{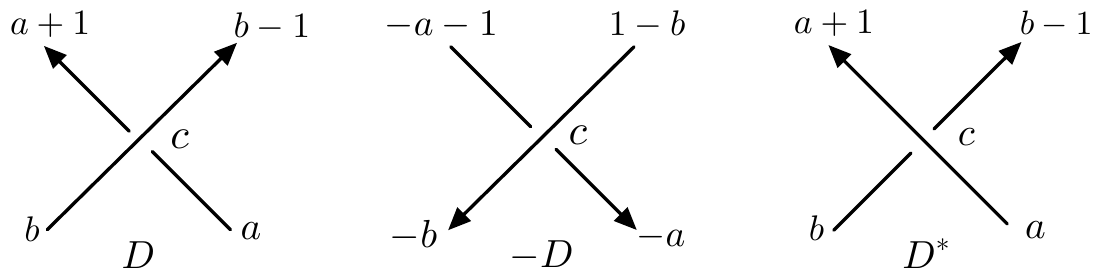}
		\caption{The crossing $c$ in $D$, $-D$ and $D^*$}
		\label{9}
	\end{figure}

	The conclusions of Lemma \ref{Lemma 3.1.} and Lemma \ref{Lemma 3.2.} are evident by Figure \ref{8} and Figure \ref{9}, hence a detailed proof is omitted here.

	A zero-height knotoid diagram $D$ is a knotoid diagram whose head and leg can be attached by an arc which does not intersect with the rest part of $D$.

	\begin{lemma}\label{Lemma 3.3.} If $D$ is a zero-height knotoid diagram, then $u(c)=0$ for each crossing $c$ in $D$. 
	\end{lemma}

	\begin{proof}
	We may attach the head and the leg of $D$ by an arc which does not intersect with the rest part of $D$, so that we obtain a classical knot diagram $\overline{D}$. It is clear that the coloring numbers of the crossings in $\overline{D}$ are the same as the corresponding ones in $D$. Then the conclusion holds by a result of \cite{12} that the coloring number of the crossings in a classical knot diagram is 0.
	\end{proof}

	Analogous to a result in \cite{10}, Proposition \ref{Proposition 3.1.} reveals the relationship between the coloring number $u(c)$ and the degree $d(c)$.

	\begin{proposition}\label{Proposition 3.1.} Let $D$ be a planar knotoid diagram, then $u(c)=d(c)\sgn(c)$ for each crossing $c$ in $D$.
	\end{proposition}

	\begin{proof} 
	The planar knotoid diagram $D$ can be regarded as the image of an immersion $f:I=[0,1]\to \mathbb{R}^2$. For an arbitrary fixed crossing $c$ of $D$, suppose that $f^{-1}(\left\lbrace c\right\rbrace )=\left\lbrace t_0, t_1\right\rbrace $, where $t_0$ corresponds to the underpassing arc and $t_1$ corresponds to the overpassing arc. Let $p=\min\left\lbrace t_0,t_1\right\rbrace $ and $q=\max\left\lbrace t_0,t_1\right\rbrace $. For an arbitrary crossing $e$ that we meet during a journey from $c$ to itself along the knotoid diagram (our trace is the loop $f([p,q])$, since $f(p)=f(q)=c$), there are 4 possibilities, and we have listed them all in Figure \ref{10}. In these cases, we can see that if the purple arc in the knotoid diagram contributes $\varepsilon$ to the Cheng coloring when it passes the crossing $e$, then the sign of the corresponding endpoint of the purple chord $e$ in the Gauss diagram is $-\varepsilon$, where $\varepsilon\in \left\lbrace -1,1\right\rbrace $. This observation indicates that the difference between the Cheng-colorings at the end and start of our journey is exactly the opposite number of the sum of the chord endpoints on the Gauss diagram that we have met. 
	
	\begin{figure}[ht]
		\centering
		\includegraphics[width=0.64\textwidth]{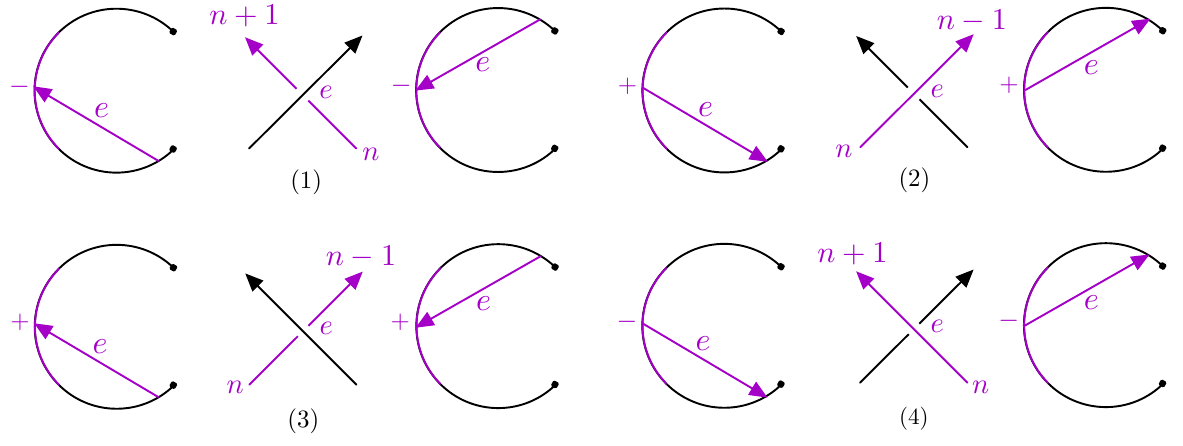}
		\caption{4 possibilities of $e$}
		\label{10}
	\end{figure}

	For the fixed crossing $c$, there are also 4 possibilities, as shown in Figure \ref{11}. In case (1), $\sgn(c)=1$ and $t_0<t_1$, which implies that in the journey "based" at $c$, we leave $c$ along the arc whose Cheng coloring is $a+1$ and return to $c$ along the arc whose Cheng coloring is $b+1$. Hence by the observation above, $(b+1)-(a+1)=-d(c)$, therefore
	\begin{center}
		$u(c)=a-b=d(c)=d(c)\sgn(c)$.
	\end{center}
	Similarly, in case (2), $\sgn(c)=1$ and $t_1<t_0$, thus 
	\begin{center}
		$u(c)=a-b=-(-d(c))=d(c)\sgn(c)$. 
	\end{center}
	In case (3), $\sgn(c)=-1$ and $t_0<t_1$, consequently
	\begin{center}
		$u(c)=a-b=-d(c)=d(c)\sgn(c)$. 
	\end{center}
	In case (4), $\sgn(c)=-1$ and $t_1<t_0$, hence $(b+1)-(a+1)=-(-d(c))$. As a result,
	\begin{center}
		$u(c)=a-b=-d(c)=d(c)\sgn(c)$. 
	\end{center}
	This completes the proof of Proposition 3.1.
	\end{proof}

	\begin{figure}[ht]
		\centering
		\includegraphics[width=0.64\textwidth]{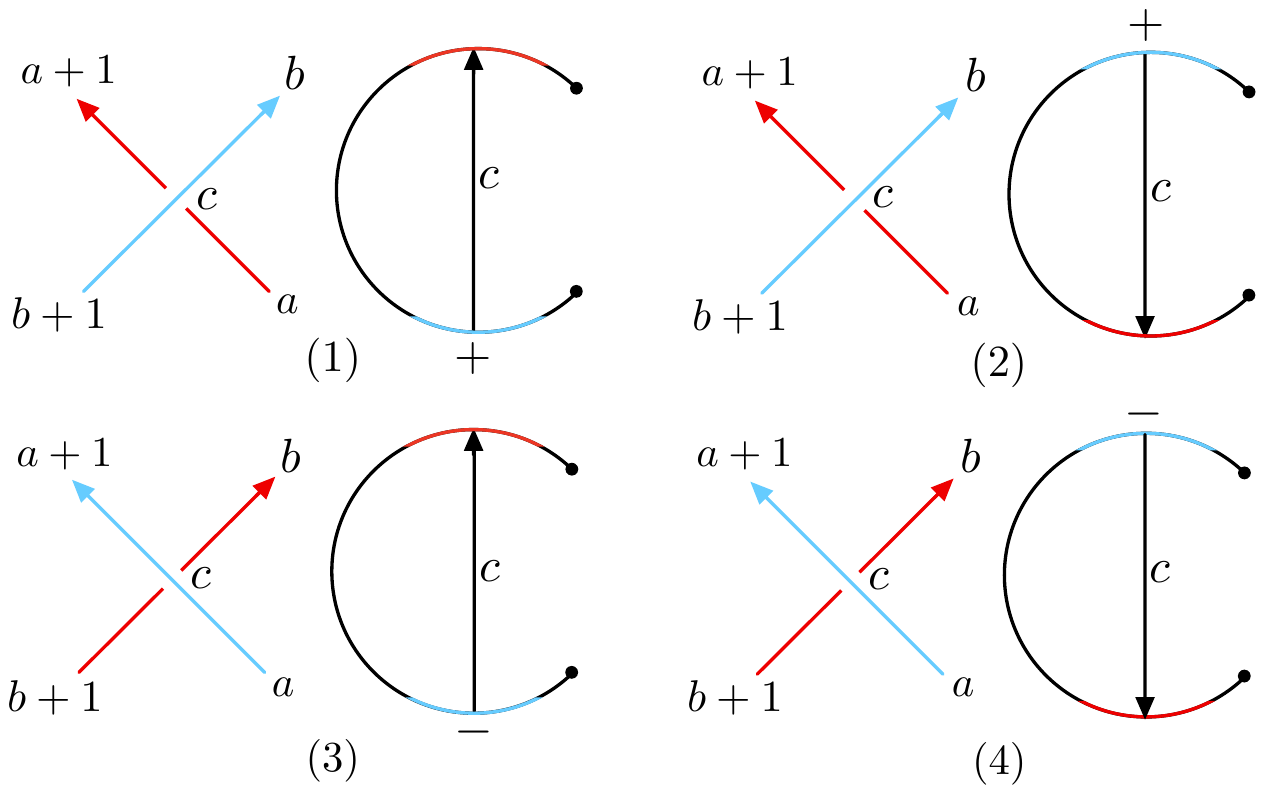}
		\caption{4 possibilities of $c$}
		\label{11}
	\end{figure}

	\begin{proposition}\label{Proposition 3.2.} Let $D$ be a planar knotoid diagram, then
		\begin{center}
			$\sum\limits_{c\in C(D)}u(c)=0$.\\
		\end{center}
	\end{proposition}

    \begin{proof}
	Let $D_1$ and $D_2$ be 2 planar knotoid diagrams which differ by a forbidden move. Suppose that $c_0$ is the crossing generated by the forbidden move. Without loss of generality, we may assume $c_0\in D_1$. Then the degree of a chord which passes through $c_0$ adds 1 or subtracts 1, depending on its direction and the sign of $c_0$. Consequently,
	\begin{center}
		$
		\begin{aligned}
			\sum\limits_{c\in r^+(c_0)}d_{D_1}(c)-\sum\limits_{c\in r^+(c_0)}d_{D_2}(c)&=\sum\limits_{c\in r^+(c_0)}(-\sgn(c_0))=-\sgn(c_0)|r^+(c_0)|,\\
			\sum\limits_{c\in r^-(c_0)}d_{D_1}(c)-\sum\limits_{c\in r^-(c_0)}d_{D_2}(c)&=\sum\limits_{c\in r^-(c_0)}(-\sgn(c_0))=-\sgn(c_0)|r^-(c_0)|,\\
			\sum\limits_{c\in l^+(c_0)}d_{D_1}(c)-\sum\limits_{c\in l^+(c_0)}d_{D_2}(c)&=\sum\limits_{c\in l^+(c_0)}\sgn(c_0)=\sgn(c_0)|l^+(c_0)|,\\
			\sum\limits_{c\in
				l^-(c_0)}d_{D_1}(c)-\sum\limits_{c\in l^-(c_0)}d_{D_2}(c)&=\sum\limits_{c\in l^-(c_0)}\sgn(c_0)=\sgn(c_0)|l^-(c_0)|.\\
		\end{aligned}
		$
	\end{center}

	\begin{figure}[ht]
		\centering
		\includegraphics[width=0.24\textwidth]{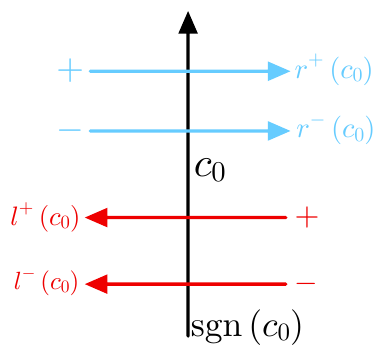}
		\caption{$c_0$ in $G(D_1)$}
		\label{12}
	\end{figure}

	Moreover, by Proposition \ref{Proposition 3.1.},
	\begin{center}
		$
		\begin{aligned}
			\sum\limits_{c\in r^+(c_0)}u_{D_1}(c)-\sum\limits_{c\in r^+(c_0)}u_{D_2}(c)&=-\sgn(c_0)|r^+(c_0)|,\\
			\sum\limits_{c\in r^-(c_0)}u_{D_1}(c)-\sum\limits_{c\in r^-(c_0)}u_{D_2}(c)&=\sgn(c_0)|r^-(c_0)|,\\
			\sum\limits_{c\in l^+(c_0)}u_{D_1}(c)-\sum\limits_{c\in l^+(c_0)}u_{D_2}(c)&=\sgn(c_0)|l^+(c_0)|,\\
			\sum\limits_{c\in l^-(c_0)}u_{D_1}(c)-\sum\limits_{c\in l^-(c_0)}u_{D_2}(c)&=-\sgn(c_0)|l^-(c_0)|.\\
		\end{aligned}
		$
	\end{center}

	It is evident that $u_{D_1}(c)=u_{D_2}(c)$ for a chord $c$ which does not cross $c_0$. As a result,
	\begin{center}
		$
		\begin{aligned}
			&\sum\limits_{c\in C(D_1)}u_{D_1}(c)-\sum\limits_{c\in C(D_2)}u_{D_2}(c)-u(c_0)\\
			=&\sum\limits_{c\in C(D_2)}u_{D_1}(c)-\sum\limits_{c\in C(D_2)}u_{D_2}(c) \\
			=&\left( \sum\limits_{c\in r^+(c_0)}u_{D_1}(c)-\sum\limits_{c\in r^+(c_0)}u_{D_2}(c)\right)+\left(\sum\limits_{c\in r^-(c_0)}u_{D_1}(c)-\sum\limits_{c\in r^-(c_0)}u_{D_2}(c)\right) \\
			+&\left( \sum\limits_{c\in l^+(c_0)}u_{D_1}(c)-\sum\limits_{c\in l^+(c_0)}u_{D_2}(c)\right) +\left( \sum\limits_{c\in l^-(c_0)}u_{D_1}(c)-\sum\limits_{c\in l^-(c_0)}u_{D_2}(c)\right) \\
			=&-\sgn(c_0)|r^+(c_0)|+\sgn(c_0)|r^-(c_0)|+\sgn(c_0)|l^+(c_0)|-\sgn(c_0)|l^-(c_0)|\\
			=&-\sgn(c_0)\left(|r^+(c_0)|-|r^-(c_0)|-|l^+(c_0)|+|l^-(c_0)|\right)\\
			=&-\sgn(c_0)d(c_0)\\
			=&-u(c_0).    
		\end{aligned}
		$
	\end{center}
	Hence
	\begin{center}
		$\sum\limits_{c\in C(D_1)}u_{D_1}(c)=\sum\limits_{c\in C(D_2)}u_{D_2}(c)$.
	\end{center}

	Then the conclusion holds obviously by the formula above, since $D$ can be deformed into a trivial knotoid diagram by a finite sequence of forbidden moves, and there is no crossing in a trivial knotoid diagram. This completes the proof of Proposition 3.2.
	\end{proof}

	Our aim is to construct invariants of knotoids of the form 
	\begin{center}
		$\sum\limits_{c\in C(D)} u(c)f_c(x_1,\cdots,x_n)$, 
	\end{center}
	where $f_c(x_1,\cdots,x_n)$ is a function defined at each crossing $c\in C(D)$. The definition given below can be used in this kind of invariants to construct more invariants.\\

	\begin{definition}
	\label{Definition 3.1.} \textbf{(Coloring-allowed function)}  Let $D$ be a knotoid diagram. Suppose that $f_c(x_1,\cdots,x_n)$ is an $n$-variable function defined at each crossing $c\in C(D)$, $n\in \mathbb{N}$. Then the function $f_c(x_1,\cdots,x_n)$ is called a \textbf{coloring-allowed function} if it satisfies the conditions below:\\
	(1)For an arbitrary $\Omega_2 a$ move operated on $D$,
	\begin{center}
		$f_{c_1}(x_1,\cdots,x_n)=f_{c_2}(x_1,\cdots,x_n)$,  
	\end{center}
	where $c_1$ and $c_2$ are the crossings generated by the $\Omega_2 a$ move;\\
	(2)For an arbitrary $\Omega_3 a$ move operated on $D$,
	\begin{center}
		$f_{c_i,D}(x_1,\cdots,x_n)=f_{c_i,D'}(x_1,\cdots,x_n)$ 
	\end{center}
	for $i=1,2,3$, where $c_1$, $c_2$ and $c_3$ are the crossings related to the $\Omega_3 a$ move, $D'$ is the knotoid diagram obtained by $D$ after the $\Omega_3 a$ move;\\
	(3)$f_c(x_1,\cdots,x_n)$ preserves for an arbitrary crossing $c$ which is not involved in an $\Omega$-move.
	\end{definition}

    With the definition of coloring-allowed functions, we can define the coloring-allowed invariants of knotoids.

	\begin{theorem}\label{Theorem 3.1.} Let $D$ be a planar knotoid diagram, $f_c(x_1,\cdots,x_n)$ be a coloring-allowed function, then 
		\begin{center}
			$\mathcal{A}_D(x_1,\cdots,x_n)=\sum\limits_{c\in C(D)} u(c)f_c(x_1,\cdots,x_n)$
		\end{center}
		is an invariant of $D$. Such an invariant is called a \textbf{coloring-allowed invariant} of knotoids.
	\end{theorem}

	\begin{proof}
	In \cite{2}, Polyak proved that all oriented versions of Reidemeister moves for knot and link diagrams can be generated by a set consisting of only 4 oriented Reidemeister moves. Hence it suffices to show that $\mathcal{A}_D(x_1,\cdots,x_n)$ is invariant under the oriented Reidemeister moves $\Omega_1 a$, $\Omega_1b$, 
	$\Omega_2 a$ and $\Omega_3 a$ of planar knotoid diagrams.
	
	Suppose that $D$ and $D_1$ are two oriented knotoid diagrams which are related by an $\Omega_1a$ or $\Omega_1b$ move, and $c_0$ is the new crossing in $D_1$ generated by this move. By Lemma \ref{Lemma 3.1.}(1) and Definition \ref{Definition 3.1.}(3), we have
	\begin{center}
		$\mathcal{A}_{D_1}(x_1,\cdots,x_n)-\mathcal{A}_{D}(x_1,\cdots,x_n)=u(c_0)f_{c_0}(x_1,\cdots,x_n)=0$.
	\end{center}
	Hence $\mathcal{A}_D(x_1,\cdots,x_n)$ remains invariant under $\Omega_1 a$ and $\Omega_1b$ moves.

	Suppose that $D$ and $D_2$ are two oriented knotoid diagrams which are related by an $\Omega_2a$ move,  and $c_1$, $c_2$ are the new crossings in $D_2$ generated by this move. By Lemma \ref{Lemma 3.1.}(2) and Definition \ref{Definition 3.1.}(1)(3),
	\begin{center}
		$
		\begin{aligned}
			&\mathcal{A}_{D_2}(x_1,\cdots,x_n)-\mathcal{A}_{D}(x_1,\cdots,x_n)=u(c_1)f_{c_1}(x_1,\cdots,x_n)+u(c_2)f_{c_2}(x_1,\cdots,x_n)=0.
		\end{aligned}
		$
	\end{center}
	Thus $\mathcal{A}_D(x_1,\cdots,x_n)$ remains invariant under $\Omega_2 a$ moves.

	Suppose that $D$ and $D_3$ are two oriented knotoid diagrams which are related by an $\Omega_3a$ move,  $c_1$, $c_2$ and $c_3$ are the crossings related to this move. It follows from Lemma \ref{Lemma 3.1.}(3) and Definition \ref{Definition 3.1.}(2)(3) that
	\begin{center}
		$
		\begin{aligned}
			&\mathcal{A}_{D_3}(x_1,\cdots,x_n)-\mathcal{A}_{D}(x_1,\cdots,x_n)
			=\displaystyle \sum\limits_{i=1}^3  u_{D_3}(c_i)f_{c_i,D_3}(x_1,\cdots,x_n)-\displaystyle \sum\limits_{i=1}^3  u_D(c_i)f_{c_i,D}(x_1,\cdots,x_n)=0.
		\end{aligned}
		$
	\end{center}
	Therefore $\mathcal{A}_D(x_1,\cdots,x_n)$ remains invariant under $\Omega_3 a$ moves.
	\end{proof}

	Several examples of coloring-allowed invarants will be given below.

	\begin{example}\label{Example 3.1.} Let $f_c^F(x)=x^{d(c)}$, and 
	\begin{center}
		$\mathcal{A}_D^F(x)=\sum\limits_{c\in C(D)} u(c)x^{d(c)}$.
	\end{center} 
	It can be proved by Lemma \ref{Lemma 3.1.} and Proposition \ref{Proposition 3.1.} that $f_c^F$ is a coloring-allowed function, which gives the invariance of $\mathcal{A}_D^F$.
	\end{example}

	\begin{example}\label{Example 3.2.} Let $f_c^G(x)=x^{i(c)}$, and 
	\begin{center}
		$\mathcal{A}_D^G(x)=\sum\limits_{c\in C(D)} u(c)x^{i(c)}$. 
	\end{center}
	One can easily verify that $f_c^G$ is coloring-allowed, thus $\mathcal{A}_D^G$ is an invariant of $D$.
	\end{example}

	\begin{example}\label{Example 3.3.} For a knotoid diagram $D$ and a fixed chord $c\in C(D)$, let 
	\begin{center}
		$C_n(D,c)=\left\lbrace e\in C(D)\mid e\cap c\neq\varnothing,\gcd(d(c),d(e))=n\right\rbrace $,
	\end{center}
	where $e\cap c$ is the intersection of the chords $e$ and $c$ in the Gauss diagram $G(D)$, and $\gcd(a,b)$ denotes the greatest common divisor of two integers $a$ and $b$. The set $C_n(D,c)$ can be decomposed into two subsets:
	\begin{center}
		$l^n(c)=l(c)\cap C_n(D,c)$;\\ 
		$r^n(c)=r(c)\cap C_n(D,c)$.
	\end{center}
	With these two subsets, for $c\in C(D)$ and $n\in \mathbb{N}$, the $n$th index function is defined by
	\begin{center}
		$g_c^n(v)=\sum_{e\in r^n(c)}\sgn(e)v^{\phi_c(d(e))}-\sum_{e\in l^n(c)}\sgn(e)v^{\phi_c(-d(e))}$.
	\end{center}
	One may refer to \cite{9} for details, where the notions mentioned above have been illustrated thoroughly.  Let
	\begin{center}
		$f_c^H(x,y,v)=\sum\limits_{n\in\mathbb{N}} x^{g_c^n(v)} y^n$.
	\end{center}
	The function $f_c^H(x,y,v)$ is a coloring-allowed function, since it has been proved in \cite{9} that $g_c^n(v)$ is a coloring-allowed function for each $n\in \mathbb{N}$. It follows that
	\begin{center}
		$
		\begin{aligned}
			\mathcal{A}_D^H(x,y,v)=\sum\limits_{c\in C(D)}u(c)f_c^H(x,y,v)&=\sum\limits_{c\in C(D)}u(c)\sum\limits_{n\in\mathbb{N}} x^{g_c^n(v)} y^n=\sum\limits_{c\in C(D)}\sum\limits_{n\in\mathbb{N}}u(c) x^{g_c^n(v)} y^n
		\end{aligned}
		$
	\end{center}
	is an invariant of $D$ by Theorem \ref{Theorem 3.1.}.
	\end{example}

	\begin{example}\label{Example 3.4.} For a knotoid diagram $D$, a fixed chord $c\in C(D)$ and $n\in\mathbb{N}$, let
	\begin{center}
		$
		\begin{aligned}
			r_n^+(c)&=r^+(c)\cap C_n(D,c)=\left\lbrace c'\in r(c)\mid\sgn(c')=1,\gcd(d(c),d(c'))=n\right\rbrace; \\
			r_n^-(c)&=r^-(c)\cap C_n(D,c)=\left\lbrace c'\in r(c)\mid\sgn(c')=-1,\gcd(d(c),d(c'))=n\right\rbrace; \\
			l_n^+(c)&=l^+(c)\cap C_n(D,c)=\left\lbrace c'\in l(c)\mid\sgn(c')=1,\gcd(d(c),d(c'))=n\right\rbrace; \\
			l_n^-(c)&=l^-(c)\cap C_n(D,c)=\left\lbrace c'\in l(c)|\mid\sgn(c')=-1,\gcd(d(c),d(c'))=n\right\rbrace .
		\end{aligned}
		$
	\end{center}
	The $n$th degree of the chord $c$ 
	is defined by
	\begin{center}
		$d_n(c)=|r_n^+(c)|-|r_n^-(c)|-|l_n^+(c)|+|l_n^-(c)|$.
	\end{center}
	It is a coloring-allowed function, and one may refer to \cite{11} for a detailed proof. Let
	\begin{center}
		$f_c^M(x_0,x_1,\cdots)=\prod_{n=0}^{\infty}x_n^{d_n(c)}$,
	\end{center}
	then $f_c^M$ is a coloring-allowed function as well. There are finite crossings in a knotoid diagram, hence $d_n(c)=0$ when $n$ is sufficiently large, consequently $x_n^{d_n(c)}=1$, which guarantees the convergence of the product given above. Moreover, the number of variables in $f_c^M$ can be considered to be finite by this fact. Therefore
	\begin{center}
		$\mathcal{A}_D^M(x_0,x_1,\cdots)=\sum\limits_{c\in C(D)}u(c)f_c^M(x_0,x_1,\cdots)=\sum\limits_{c\in C(D)}u(c)\prod_{n=0}^{\infty}x_n^{d_n(c)}$
	\end{center}
	is an invariant of $D$ by Theorem \ref{Theorem 3.1.}.
	\end{example}

	We will calculate the coloring-allowed invariants given in Examples 3.1-3.4 of a knotoid diagram in Example \ref{Example 3.5.}.

	\begin{example}\label{Example 3.5.} Let $D$ be the planar knotoid diagram shown in Figure \ref{13}. Then a direct calculation gives

	\begin{figure}[ht]
		\centering
		\includegraphics[width=0.48\textwidth]{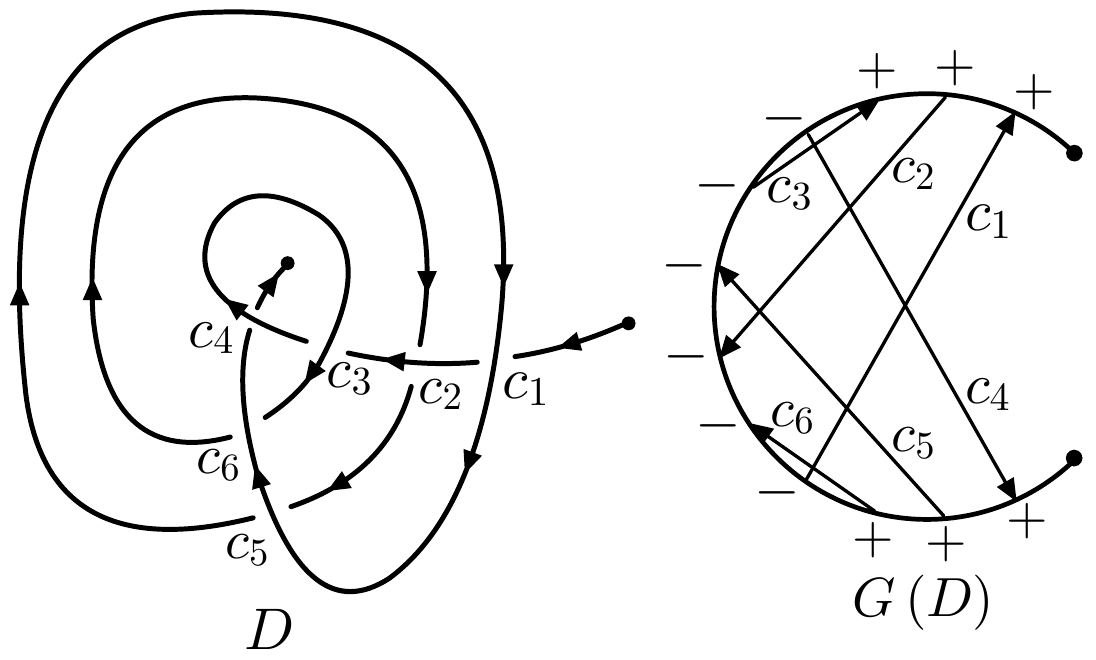}
		\caption{$D$ and $G(D)$}
		\label{13}
	\end{figure}

	\begin{table}[h!]
		\centering
		\renewcommand{\arraystretch}{1.4}
		\begin{tabular}{|c|c|c|c|c|c|c|}
			\hline
			&$c_1$&$c_2$&$c_3$&$c_4$&$c_5$&$c_6$\\
			\hline
			$u(c)$&3&2&1&$-3$&$-2$&$-1$\\
			\hline
			$d(c)$&$-3$&2&$-1$&3&$-2$&$-1$\\
			\hline
			$d_1(c)$&$-2$&1&$-1$&2&$-1$&$-1$\\
			\hline
			$d_2(c)$&0&1&0&0&$-1$&0\\
			\hline
			$d_3(c)$&$-1$&0&0&1&0&0\\
			\hline
			$i(c)$&3&2&1&3&2&1\\
			\hline
			$g_c^1(v)$&$-v^2-v$&$v$&$-1$&$v^2+v$&$-v$&$-1$\\
			\hline
			$g_c^2(v)$&0&1&0&0&$-1$&0\\
			\hline
			$g_c^3(v)$&$-1$&0&0&1&0&0\\
			\hline
		\end{tabular}
		\caption{The knotoid diagram $D$ in Figure \ref{13}}
		\label{Table 1}
	\end{table}

	As a result,
	
	\begin{center}
		$
		\begin{aligned}
			&\mathcal{A}_D^F(x)=-3x^3+2x^2-2x^{-2}+3x^{-3},\\
			&\mathcal{A}_D^G(x)=0,\\
			&\mathcal{A}_D^H(x,y,v)=3(x^{-1}y^3+x^{-v^2-v}y)+2(xy^2+x^vy)-3(xy^3+x^{v^2+v}y)-2(x^{-1}y^2+x^{-v}y),\\
			&\mathcal{A}_D^M(x_1,x_2,x_3)=-3x_1^2x_3+2x_1x_2-2x_1^{-1}x_2^{-1}+3x_1^{-2}x_3^{-1}.\\
		\end{aligned}
		$
	\end{center}

	\end{example}

	One may find that the notions involved in the definition of coloring-allowed invariants can be generalized to classical knots, thus it seems natural to generalize the notion of coloring-allowed invariants itself to classical knots as well. However, this is not the case since every coloring-allowed invariant becomes trivial for zero-height knotoid diagrams. Proposition \ref{Proposition 3.3.} is evident by Lemma \ref{Lemma 3.3.}.

	\begin{proposition}\label{Proposition 3.3.} If $D$ is a zero-height knotoid diagram and $\mathcal{A}_D$ is a coloring allowed invariant, then $\mathcal{A}_D$ is trivial, namely $\mathcal{A}_D\equiv 0$.\\
	\end{proposition}

	\section{The 4-phases functions}

	In this section, we will introduce 4 coloring-alowed invariants, which are called the 4-phases functions, and illustrate several properties of them. \\
	
	\noindent\textbf{Definition 4.1. (The 4-phases functions)} Let $D$ be a planar knotoid diagram and $c\in C(D)$. \\
	Let $f_c^B(x,v)=x^{g_c(v)}$, then the \textbf{$B$-function} of $D$ is defined by
	\begin{center}
		$B_D(x,v)=\mathcal{A}_D^B(x,v)=\displaystyle \sum\limits_{c\in C(D)}  u(c)f_c^B(x,v)=\displaystyle \sum\limits_{c\in C(D)}  u(c)x^{g_c(v)}$.
	\end{center}
	Let $f_c^C(x,v)=x^{h_c(v)}$, then the \textbf{$C$-function} of $D$ is defined by
	\begin{center}
		$C_D(x,v)=\mathcal{A}_D^C(x,v)=\displaystyle \sum\limits_{c\in C(D)}  u(c)f_c^C(x,v)=\displaystyle \sum\limits_{c\in C(D)}  u(c)x^{h_c(v)}$.
	\end{center}
	Let $f_c^{\widetilde{B}}(x,v)=x^{\widetilde{g}_c(v)}$, then the \textbf{$\widetilde{B}$-function} of $D$ is defined by
	\begin{center}
		$\widetilde{B}_D(x,v)=\mathcal{A}_D^{\widetilde{B}}(x,v)=\displaystyle \sum\limits_{c\in C(D)}  u(c)f_c^{\widetilde{B}}(x,v)=\displaystyle \sum\limits_{c\in C(D)}  u(c)x^{\widetilde{g}_c(v)}$.
	\end{center}
	Let $f_c^{\widetilde{C}}(x,v)=x^{\widetilde{h}_c(v)}$, then the \textbf{$\widetilde{C}$-function} of $D$ is defined by
	\begin{center}
		$\widetilde{C}_D(x,v)=\mathcal{A}_D^{\widetilde{C}}(x,v)=\displaystyle \sum\limits_{c\in C(D)}  u(c)f_c^{\widetilde{C}}(x,v)=\displaystyle \sum\limits_{c\in C(D)}  u(c)x^{\widetilde{h}_c(v)}$.
	\end{center}

	The number of crossings in a knotoid diagram is finite, hence there are only finite non-zero terms in the 4 series above, which ensures the convergence of them. These functions given above are collectively referred to as the \textbf{4-phases functions}. \\

	\begin{theorem}\label{Theorem 4.1.} Let $D$ be a planar knotoid diagram, then the 4-phases functions $B_D(x,v)$, $C_D(x,v)$,  $\widetilde{B}_D(x,v)$ and $\widetilde{C}_D(x,v)$ are invariants of $D$.
	\end{theorem}

	\begin{proof} 
		Here we prove the invariance of $B_D(x,v)$ as an instance, since the invariance of $C_D(x,v)$,  $\widetilde{B}_D(x,v)$ and $\widetilde{C}_D(x,v)$ can be verified analogously. By Theorem \ref{Theorem 3.1.}, it suffices to show that $x^{g_c(v)}$ is a coloring-allowed function. Thereby a proof showing that $g_c(v)$  satisfies (1), (2) and (3) in Definition \ref{Definition 3.1.} is enough.

	We will verify (3) first. We have seen that $u(c)$ preserves during an $\Omega$-move for a crossing $c$ which is not involved in it in Lemma \ref{Lemma 3.1.}. Obviously, the sign $\sgn(c)$ of such a crossing $c$ preserves as well. Then by Proposition \ref{Proposition 3.1.}, $d(c)$ preserves during an $\Omega$-move. As a consequence, for 2 such crossings $c$ and $e$ which satisfies $c\cap e\neq \varnothing$ for the corresponding chords in the Gauss diagram, $e$ contributes equally to $g_{c,D}(v)$ and $g_{c,D'}(v)$. Therefore we only need to consider the effect of the crossings involved in an $\Omega$-move to $g_c(v)$.

	\begin{figure}[ht]
		\centering
		\includegraphics[width=0.15\textwidth]{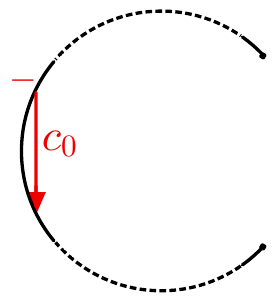}
		\caption{Gauss diagram in an $\Omega_1$ move}
		\label{14}
	\end{figure}

	In an $\Omega_1$ move, for the new crossing $c_0$ generated by this move, the corresponding chord in the Gauss diagram is isolated, as shown in Figure \ref{14}, thus $g_c(v)$ preserves during an $\Omega_1$ move. In an $\Omega_2a$ move, for the new crossings $c_1$ and $c_2$ generated by this move, the corresponding chords in the Gauss diagram pass through $c$ in the same direction with opposite signs, and $d(c_1)=d(c_2)$, hence their effects on $g_c(v)$ offset, which yields $g_{c,D}(v)=g_{c,D'}(v)$. The invariance of $g_c(v)$ during an $\Omega_3a$ move is evident.

	For (1), suppose that $D$ and $D'$ are two oriented knotoid diagrams which are related by an $\Omega_2a$ move, $G(D)$ and $G(D')$ are the corresponding Gauss diagrams respectively, as shown in Figure \ref{15}, and $c_1$, $c_2$ are the new crossings in $D'$ generated by this $\Omega_2a$ move. Clearly, $\sgn(c_1)=-1$, $\sgn(c_2)=1$, $d(c_1)=d(c_2)$, $d(c)$ preserves for a crossing $c$ in $D$ and $c$ crosses $c_1$ and $c_2$ in the same direction, therefore
	
	\begin{center}
		$
		\begin{aligned}
			g_{c_1}(v)-g_{c_2}(v)&=\sgn(c_2)v^{\phi_{c_1}(d(c_2))}-(-\sgn(c_1)v^{\phi_{c_2}(-d(c_1))})=v^0+(-1)v^0=0,
		\end{aligned}
		$
	\end{center}
	which is equivalent to $g_{c_1}(v)=g_{c_2}(v)$. 
	
	\begin{figure}[ht]
		\centering
		\includegraphics[width=0.4\textwidth]{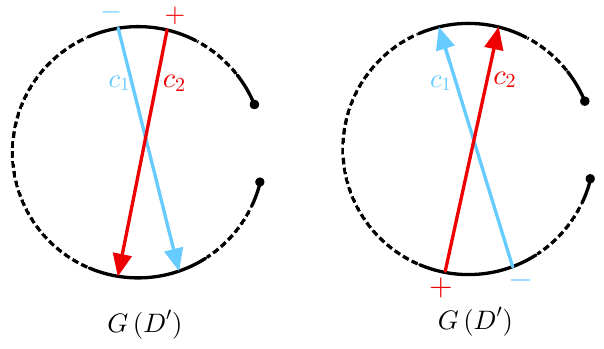}
		\caption{Gauss diagrams in an $\Omega_2a$ move}
		\label{15}
	\end{figure}

	For (2), suppose that $D$ and $D'$ are two oriented knotoid diagrams which are related by an $\Omega_3a$ move, $G(D)$ and $G(D')$ are the corresponding Gauss diagrams respectively, as shown in Figure \ref{16}. Suppose that $c_1$, $c_2$ and $c_3$ are the crossings related to this $\Omega_3a$ move. Evidently, $\sgn_D(c_i)=\sgn_{D'}(c_i)$, $d_D(c_i)=d_{D'}(c_i)$, hence for the sake of simplicity, we denote them by $\sgn(c_i)$ and $d(c_i)$ respectively, $i=1,2,3$. Denote the sum of signs of chord endpoints on the blue region by $x$ and the red region by $y$ respectively. Notice that 
	\begin{center}
		$d(c_1)+d(c_3)-d(c_2)=x+y-(x+y)=0$.
	\end{center}

	\begin{figure}[ht]
		\centering
		\includegraphics[width=0.44\textwidth]{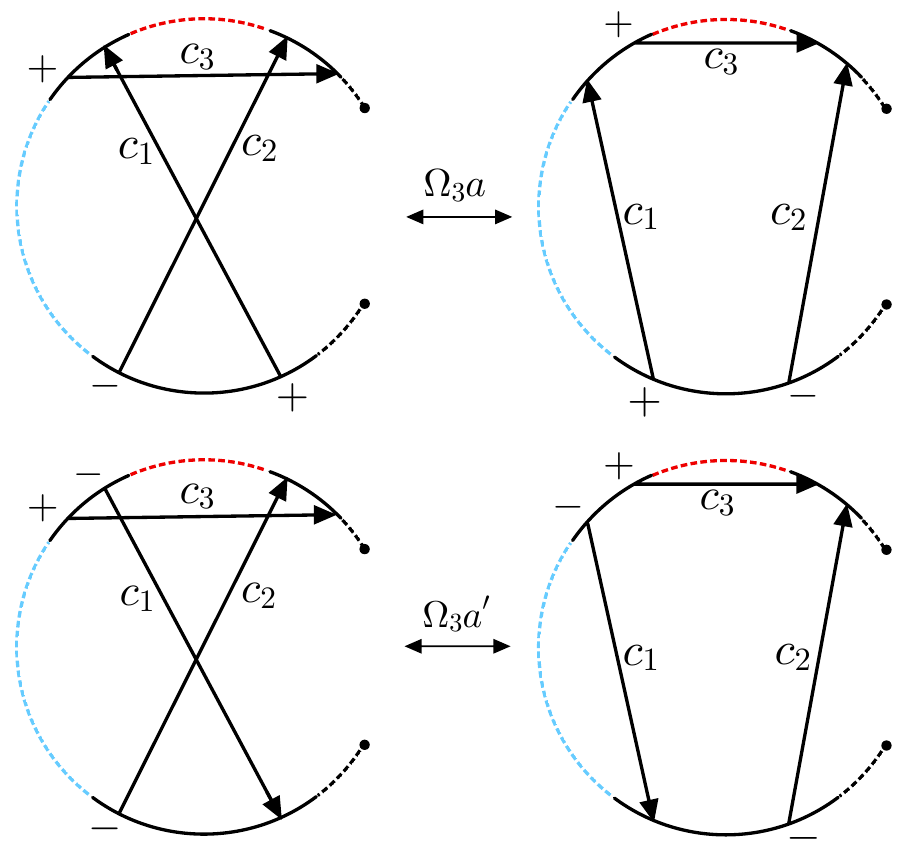}
		\caption{Gauss diagrams in an $\Omega_3a$ move}
		\label{16}
	\end{figure}

	As a result, we have
	\begin{center}
		$
		\begin{aligned}
			g_{c_1,D}(v)-g_{c_1,D'}(v)&=\sgn(c_2)v^{\phi_{c_1}(d(c_2))}+\sgn(c_3)v^{\phi_{c_1}(d(c_3))}\\
			&=-v^{\phi_{c_1}(d(c_2))}+v^{\phi_{c_1}(d(c_3))}\\
			&=-v^{\phi_{c_1}(d(c_1)+d(c_3))}+v^{\phi_{c_1}(d(c_3))}\\
			&=-v^{\phi_{c_1}(d(c_3))}+v^{\phi_{c_1}(d(c_3))}=0,\\
			g_{c_2,D}(v)-g_{c_2,D'}(v)&=-\sgn(c_1)v^{\phi_{c_2}(-d(c_1))}+\sgn(c_3)v^{\phi_{c_2}(d(c_3))}\\&=-v^{\phi_{c_2}(-d(c_1))}+v^{\phi_{c_2}(d(c_3))}\\
			&=-v^{\phi_{c_2}(-d(c_2)+d(c_3))}+v^{\phi_{c_2}(d(c_3))}\\
			&=-v^{\phi_{c_2}(d(c_3))}+v^{\phi_{c_2}(d(c_3))}=0,\\
			g_{c_3,D}(v)-g_{c_3,D'}(v)&=-\sgn(c_1)v^{\phi_{c_3}(-d(c_1))}-\sgn(c_2)v^{\phi_{c_3}(-d(c_2))}\\
			&=-v^{\phi_{c_3}(-d(c_1))}+v^{\phi_{c_3}(-d(c_2))}\\
			&=-v^{\phi_{c_3}(-d(c_2)+d(c_3))}+v^{\phi_{c_3}(-d(c_2))}\\
			&=-v^{\phi_{c_3}(-d(c_2))}+v^{\phi_{c_3}(-d(c_2))}=0.
		\end{aligned}
		$
	\end{center}
	Here we have used the fact that $\phi_c$ is a  homomorphism and $\phi_c(d(c))=0$ in the calculation above which seems a little bit tedious. As a consequence, $g_{c_i,D}(v)=g_{c_i,D'}(v)$, $i=1,2,3$. This proves the invariance of $B_D(x,v)$. 
	\end{proof}

	We will calculate the 4-phases functions of a knotoid diagram in Example \ref{Example 4.1.}.

	\begin{example}\label{Example 4.1.} Let $D$ be the planar knotoid diagram shown in Figure \ref{13}. Then a direct calculation gives

	\begin{table}[h!]
		\centering
		\renewcommand{\arraystretch}{1.4}
		\begin{tabular}{|c|c|c|c|c|c|c|}
			\hline
			&$c_1$&$c_2$&$c_3$&$c_4$&$c_5$&$c_6$\\
			\hline
			$u(c)$&3&2&1&$-3$&$-2$&$-1$\\
			\hline
			$d(c)$&$-3$&2&$-1$&3&$-2$&$-1$\\
			\hline
			$i(c)$&3&2&1&3&2&1\\
			\hline
			$g_c(v)$&$-v^2-v-1$&$v+1$&$-1$&$v^2+v+1$&$-v-1$&$-1$\\
			\hline
			$h_c(v)$&$v^2+v+1$&$v+1$&$1$&$v^2+v+1$&$v+1$&$1$\\
			\hline
			$\widetilde{g}_c(v)$&$-v^2-v-1$&$v+1$&$-1$&$2v^2+1$&$-v-1$&$-1$\\
			\hline
			$\widetilde{h}_c(v)$&$v^2+v+1$&$v+1$&$1$&$2v^2+1$&$v+1$&$1$\\
			\hline
		\end{tabular}
		\caption{The knotoid diagram $D$ in Figure \ref{13}}
		\label{Table 2}
	\end{table}

	As a result,
	
	\begin{center}
		$
		\begin{aligned}
			B_D(x,v)&=3x^{-v^2-v-1}-3x^{v^2+v+1}+2x^{v+1}-2x^{-v-1}, 
			~C_D(x,v)=0,\\
			\widetilde{B}_D(x,v)&=3x^{-v^2-v-1}-3x^{2v^2+1}+2x^{v+1}-2x^{-v-1},
			~~~\widetilde{C}_D(x,v)=3x^{v^2+v+1}-3x^{2v^2+1}.
		\end{aligned}
		$
	\end{center}

	\end{example}

	Next we discuss some properties of the 4-phases functions.

	\begin{theorem}\label{Theorem 4.2.} Let $D$ be a knotoid diagram. If $-D$ is the inverse of $D$, then\\ (1)$B_{-D}(x,v)=-B_D(x^{-1},v)$, $C_{-D}(x,v)=-C_D(x,v^{-1})$,\\     (2)$\widetilde{B}_{-D}(x,v)=-\widetilde{B}_D(x^{-1},v^{-1})$, $\widetilde{C}_{-D}(x,v)=-\widetilde{C}_D(x,v)$.
	\end{theorem}

	\begin{proof} Suppose that $G(D)$ and $G(-D)$ are the corresponding Gauss diagrams respectively, as shown in Figure \ref{18}. By the definition of the inverse of a knotoid diagram, $G(-D)$ is the Gauss diagram obtained by rotating $G(D)$ around the horizontal axis $180^\circ$, hence $\forall c\in C(D)$, we have $i_{-D}(c)=i_D(c)$, $\lambda_{-D}(c)=\lambda_D(c)$ and $\rho_{-D}(c)=\rho_D(c)$. Obviously, $\sgn_{-D}(c)=\sgn_D(c)$. The observation above implies that $\widetilde{h}_{c,-D}(v)=\widetilde{h}_{c,D}(v)$ for each chord $c$ in $G(D)$. By Lemma \ref{Lemma 3.2.}, $u_{-D}(c)=-u_D(c)$. As a result,
	
	\begin{center}
		$
		\begin{aligned}
			\widetilde{C}_{-D}(x,v)&=\displaystyle \sum\limits_{c\in C(-D)}  u_{-D}(c)x^{\widetilde{h}_{c,-D}(v)}=\displaystyle \sum\limits_{c\in C(D)}  -u_D(c)x^{\widetilde{h}_{c,D}(v)}=-\widetilde{C}_D(x,v).\\
		\end{aligned}
		$\\
	\end{center}

	\begin{figure}[ht]
		\centering
		\includegraphics[width=0.88\textwidth]{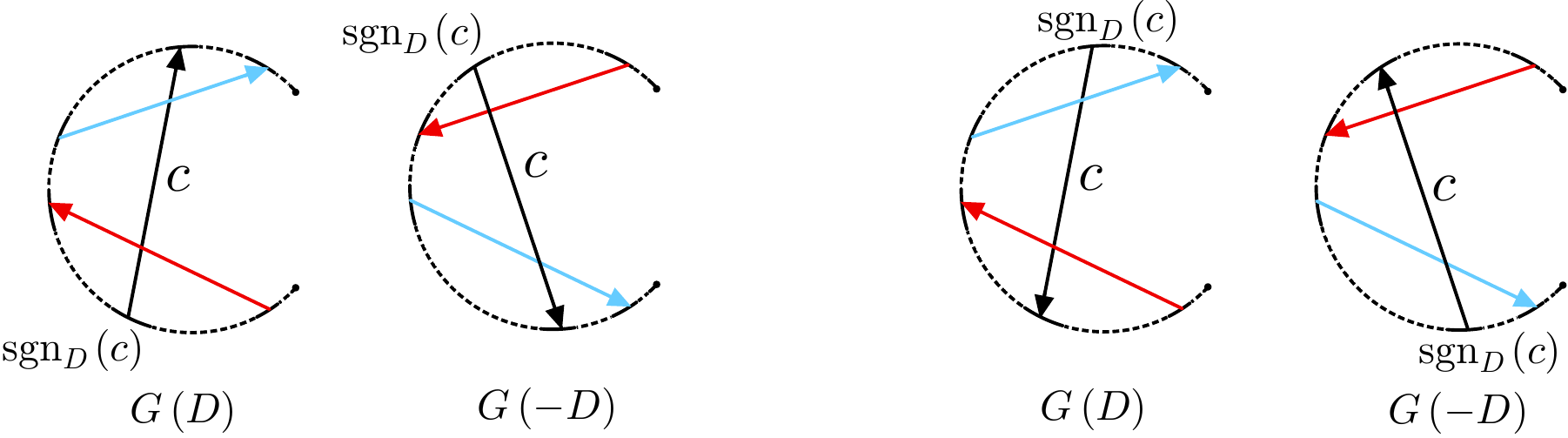}
		\caption{$G(D)$ and $G(-D)$}
		\label{18}
	\end{figure}

	For $C_{-D}(x,v)$, since $d_{-D}(e)=-d_D(e)$ for an arbitrary $e\in C(D)$, which can be seen in Figure \ref{18}, it follows that
	\begin{center}
		$
		\begin{aligned}
			h_{c,-D}(v)&=\sum_{e\in\rho_{-D}(c)}\sgn_{-D}(e)v^{\phi_c(d_{-D}(e))}-\sum_{e\in\lambda_{-D}(c)}\sgn_{-D}(e)v^{\phi_c(-d_{-D}(e))}\\
			&=\sum_{e\in\rho_{D}(c)}\sgn_{D}(e)v^{\phi_c(-d_{D}(e))}-\sum_{e\in\lambda_{D}(c)}\sgn_{D}(e)v^{\phi_c(d_{D}(e))}\\
			&=\sum_{e\in\rho_{D}(c)}\sgn_{D}(e)(v^{-1})^{\phi_c(d_{D}(e))}-\sum_{e\in\lambda_{D}(c)}\sgn_{D}(e)(v^{-1})^{\phi_c(-d_{D}(e))}\\
			&=h_{c,D}(v^{-1})
		\end{aligned}
		$
	\end{center}
	for all chords in $G(D)$. Consequently,
	\begin{center}
		$
		\begin{aligned}
			C_{-D}(x,v)&=\displaystyle \sum\limits_{c\in C(-D)}  u_{-D}(c)x^{{h}_{c,-D}(v)}=\displaystyle \sum\limits_{c\in C(D)}  -u_D(c)x^{{h}_{c,D}(v^{-1})}=-{C}_D(x,v^{-1}).\\
		\end{aligned}
		$\\
	\end{center}

	If the lead side of $c$ in $G(D)$ is the right hand side, and the lead side of $c$ in $G(-D)$ is the left hand side, as shown in the left half of Figure \ref{18}, then $g_{c,-D}(v)=h_{c,-D}(v)=h_{c,D}(v^{-1})=-g_{c,D}(v)$ and $\widetilde{g}_{c,-D}(v)=\widetilde{h}_{c,-D}(v)=\widetilde{h}_{c,D}(v)=-\widetilde{g}_{c,D}(v^{-1})$. Otherwise, if the lead side of $c$ in $G(D)$ is the left hand side, and the lead side of $c$ in $G(-D)$ is the right hand side, as shown in the right half of Figure \ref{18}, then $g_{c,-D}(v)=-h_{c,-D}(v^{-1})=-h_{c,D}(v)=-g_{c,D}(v)$
	and	
	$\widetilde{g}_{c,-D}(v)=-\widetilde{h}_{c,-D}(v^{-1})=-\widetilde{h}_{c,D}(v^{-1})=-\widetilde{g}_{c,D}(v^{-1})$.
	Different cases lead to the same conclusion. As a consequence,

	\begin{center}
		$
		\begin{aligned}
			B_{-D}(x,v)&=\displaystyle \sum\limits_{c\in C(-D)}  u_{-D}(c)x^{g_{c,-D}(v)}=\displaystyle \sum\limits_{c\in C(D)}  -u_D(c)x^{-g_{c,D}(v)}=-B_D(x^{-1},v)\\
		\end{aligned}
		$\\
	\end{center}
	and
	\begin{center}
		$
		\begin{aligned}
			\widetilde{B}_{-D}(x,v)&=\displaystyle \sum\limits_{c\in C(-D)}  u_{-D}(c)x^{\widetilde{g}_{c,-D}(v)}=\displaystyle \sum\limits_{c\in C(D)}  -u_D(c)x^{-\widetilde{g}_{c,D}(v^{-1})}=-\widetilde{B}_D(x^{-1},v^{-1}).\\
		\end{aligned}
		$
	\end{center}
	
	\end{proof}

	By Theorem \ref{Theorem 4.2.}, the following corollary holds immediately:
	
	\begin{corollary}\label{Corollary 4.1.} If $D$ is invertible, then\\ (1)$B_{D}(x,v)=-B_D(x^{-1},v)$, $C_{D}(x,v)=-C_D(x,v^{-1})$,\\  (2)$\widetilde{B}_{D}(x,v)=-\widetilde{B}_D(x^{-1},v^{-1})$, $\widetilde{C}_{D}(x,v)=0$. \\
	\end{corollary}

	\begin{theorem}\label{Theorem 4.3.} Let $D$ be a knotoid diagram. If $D^*$ is the mirror image of $D$, then\\ (1)$B_{D^*}(x,v)=B_D(x^{-1},v^{-1})$, $C_{D^*}(x,v)=C_D(x,v)$,\\  (2)$\widetilde{B}_{D^*}(x,v)=\widetilde{B}_D(x^{-1},v)$, $\widetilde{C}_{D^*}(x,v)=\widetilde{C}_D(x,v^{-1})$.
	\end{theorem}

	\begin{proof} Suppose that $G(D)$ and $G(D^*)$ are the corresponding Gauss diagrams respectively, as shown in Figure \ref{19}. According to the definition of the mirror image of a knotoid diagram, $G(D^*)$ is the Gauss diagram obtained by changing the direction of each chord in $G(D)$. By the definition of $D^*$, clearly $\sgn_{D^*}(c)=-\sgn_D(c)$ for each crossing $c$, hence $\forall c\in C(D)$, $i_{D^*}(c)=i_D(c)$. Moreover, $r_{D^*}(c)=r_D(c)$ and $l_{D^*}(c)=l_D(c)$. It can be deduced from the four equalities above that 
	$\widetilde{g}_{c,D^*}(v)=-\widetilde{g}_{c,D}(v)$ for each chord $c$ in $G(D)$. By Lemma \ref{Lemma 3.2.}, $u_{D^*}(c)=u_D(c)$. As a consequence,

	\begin{center}
		$
		\begin{aligned}
			\widetilde{B}_{D^*}(x,v)&=\displaystyle \sum\limits_{c\in C(D^*)}  u_{D^*}(c)x^{\widetilde{g}_{c,D^*}(v)}=\displaystyle \sum\limits_{c\in C(D)}  u_D(c)x^{-\widetilde{g}_{c,D}(v)}=\widetilde{B}_D(x^{-1},v).\\
		\end{aligned}
		$\\
	\end{center}
	
	\begin{figure}[ht]
		\centering
		\includegraphics[width=0.88\textwidth]{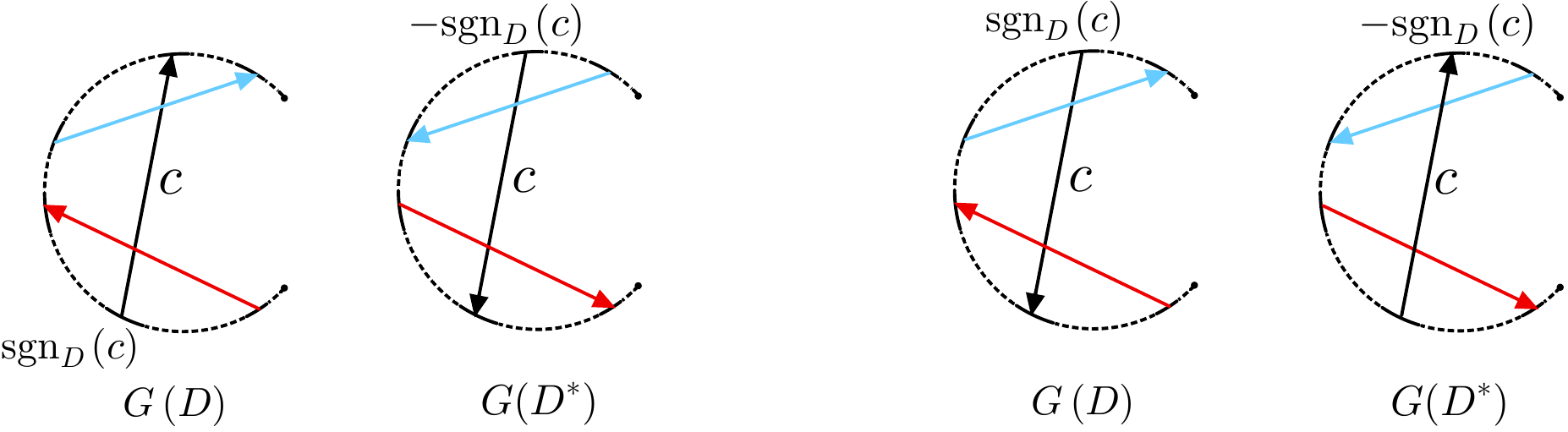}
		\caption{$G(D)$ and $G(D^*)$}
		\label{19}
	\end{figure}

	For $B_{D^*}(x,v)$, since $d_{D^*}(e)=-d_D(e)$ for an arbitrary $e\in C(D)$, which can be seen in Figure \ref{19}, it follows that
	\begin{center}
		$
		\begin{aligned}
			g_{c,D^*}(v)&=\sum_{e\in r(D^*)}\sgn_{D^*}(e)v^{\phi_c(d_{D^*}(e))}-\sum_{e\in l(D^*)}\sgn_{D^*}(e)v^{\phi_c(-d_{D^*}(e))}\\
			&=\sum_{e\in r(D)}-\sgn_{D}(e)v^{\phi_c(-d_{D}(e))}-\sum_{e\in l(D)}-\sgn_{D}(e)v^{\phi_c(d_D(e))}\\
			&=-\left( \sum_{e\in r(D)}\sgn_{D}(e)(v^{-1})^{\phi_c(d_{D}(e))}-\sum_{e\in l(D)}\sgn_{D}(e)(v^{-1})^{\phi_c(-d_D(e))}\right) \\
			&=-g_{c,D}(v^{-1})
		\end{aligned}
		$
	\end{center}
	for all chords in $G(D)$. Consequently,
	\begin{center}
		$
		\begin{aligned}
			B_{D^*}(x,v)&=\displaystyle \sum\limits_{c\in C(D^*)}  u_{D^*}(c)x^{{g}_{c,D^*}(v)}=\displaystyle \sum\limits_{c\in C(D)}  u_D(c)x^{-{g}_{c,D}(v^{-1})}={B}_D(x^{-1},v^{-1}).\\
		\end{aligned}
		$\\
	\end{center}

	If the lead side of $c$ in $G(D)$ is the right hand side, and the lead side of $c$ in $G(D^*)$ is the left hand side, as shown in the left half of Figure \ref{19}, then
	${h}_{c,D^*}(v)={g}_{c,D^*}(v)=-{g}_{c,D}(v^{-1})={h}_{c,D}(v)$
	and
	$\widetilde{h}_{c,D^*}(v)=\widetilde{g}_{c,D^*}(v)=-\widetilde{g}_{c,D}(v)=\widetilde{h}_{c,D}(v^{-1})$.
	Otherwise, if the lead side of $c$ in $G(D)$ is the left hand side, and the lead side of $c$ in $G(D^*)$ is the right hand side, as shown in the right half of Figure \ref{19}, then
	${h}_{c,D^*}(v)=-{g}_{c,D^*}(v^{-1})={g}_{c,D}(v)={h}_{c,D}(v)$
	and
	$\widetilde{h}_{c,D^*}(v)=-\widetilde{g}_{c,D^*}(v^{-1})=\widetilde{g}_{c,D}(v^{-1})=\widetilde{h}_{c,D}(v^{-1})$.
	Different cases lead to the same conclusion.  As a result, 
	\begin{center}
		$
		\begin{aligned}
			{C}_{D^*}(x,v)&=\displaystyle \sum\limits_{c\in C(D^*)}  u_{D^*}(c)x^{{h}_{c,D^*}(v)}=\displaystyle \sum\limits_{c\in C(D)}  u_D(c)x^{{h}_{c,D}(v)}={C}_D(x,v)\\
		\end{aligned}
		$\\
	\end{center}
	and
	\begin{center}
		$
		\begin{aligned}
			\widetilde{C}_{D^*}(x,v)&=\displaystyle \sum\limits_{c\in C(D^*)}  u_{D^*}(c)x^{\widetilde{h}_{c,D^*}(v)}=\displaystyle \sum\limits_{c\in C(D)}  u_D(c)x^{\widetilde{h}_{c,D}(v^{-1})}=\widetilde{C}_D(x,v^{-1}).
		\end{aligned}
		$\\
	\end{center}
	
	\end{proof}

	\begin{corollary}\label{Corollary 4.2.} If $D$ is amphicheiral, then\\ (1)$B_{D}(x,v)=B_D(x^{-1},v^{-1})$,\\    (2)$\widetilde{B}_{D}(x,v)=\widetilde{B}_D(x^{-1},v)$, $\widetilde{C}_{D}(x,v)=\widetilde{C}_D(x,v^{-1})$.
	\end{corollary}

	\noindent\textbf{Example 4.2.} By the result of Example \ref{Example 4.1.}, we can see that  $\widetilde{C}_{D}(x,v)\neq 0$, hence the planar knotoid diagram $D$ shown in Figure \ref{19} is not invertible by Corollary \ref{Corollary 4.2.}. Again by Example \ref{Example 4.1.}, we have $\widetilde{B}_{D}(x,v)\neq \widetilde{B}_D(x^{-1},v)$, which implies the cheirality of $D$.\\

	Unlike the crossings that we have met so far, some crossings in a knotoid diagram may not be endowed with over/under-crossing information, which are called singular crossings, are decorated with a black dot. An oriented singular planar knotoid diagram is an oriented planar knotoid diagram with finite singular crossings. Two oriented singular planar knotoid diagrams are said to be equivalent if one can be obtained from the other by a finite sequence of oriented Reidemeister moves, singular Reidemeister moves and planar isotopies. The corresponding equivalence classes are called singular planar knotoids.

	Let $D(c)$ be an oriented singular planar knotoid diagram with a singular crossing $c$. Denote the planar knotoid diagram obtained by replacing $c$ by a positive crossing and a negative crossing by $D(c^+)$ and $D(c^-)$ respectively, as shown in Figure \ref{20}. Any planar knotoid invariant $v$ can be extended to an invariant of singular planar knotoids using the Vassiliev skein relation:
	\begin{center}
		$v(K)=v(D(c))=v(D(c^+))-v(D(c^-))$,
	\end{center}
	where $D(c)$ is a diagram of $K$ and $c$ is a singular crossing of $D(c)$. $v(K)$ does not depend on the choice of $D(c)$ since $v$ is a planar knotoid invariant. A planar knotoid invariant $v$ is called a Vassiliev invariant of order $n$ if $v(K)\neq 0$ for a singular planar knotoid $K$ with $n$ singular crossings and $v$ vanishes for singular planar knotoids with more than $n$ singular crossings. One may refer to \cite{16}, \cite{9} or \cite{14} for details.

	\begin{figure}[ht]
		\centering
		\includegraphics[width=0.42\textwidth]{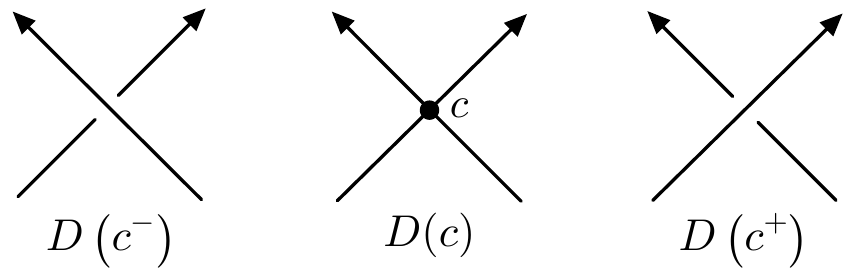}
		\caption{$D(c^+)$ and $D(c^-)$}
		\label{20}
	\end{figure}

	\begin{theorem}\label{Theorem 4.4.} Let $K$ be a planar knotoid, then $B_K(x,v)$ is a Vassiliev invariant of order 1.
	\end{theorem}

	\begin{proof} Suppose that $D(c_1,c_2)$ is a singular planar knotoid diagram with 2 singular crossings $c_1$ and $c_2$. For convenience, let $D^{++}=D(c_1^+,c_2^+)$, $D^{+-}=D(c_1^+,c_2^-)$, $D^{-+}=D(c_1^-,c_2^+)$ and $D^{--}=D(c_1^-,c_2^-)$.  Since $u(c_i^+)=u(c_i^-)$, denote $u(c_i)=u(c_i^+)=u(c_i^-)$ for $i=1,2$. Then a direct calculation gives
	\begin{center}
		$
		\begin{aligned}
			&B_{D(c_1,c_2)}(x,v)\\
			=&B_{D(c_1,c_2^+)}(x,v)-B_{D(c_1,c_2^-)}(x,v)\\
			=&B_{D(c_1^+,c_2^+)}(x,v)-B_{D(c_1^-,c_2^+)}(x,v)-B_{D(c_1^+,c_2^-)}(x,v)+B_{D(c_1^-,c_2^-)}(x,v)\\
			=&B_{D^{++}}(x,v)-B_{D^{-+}}(x,v)-B_{D^{+-}}(x,v)+B_{D^{--}}(x,v)\\
			=&\sum\limits_{c\in C(D^{++})}u(c)x^{g_{c,D^{++}}(v)}-\sum\limits_{c\in C(D^{-+})}u(c)x^{g_{c,D^{-+}}(v)}-\sum\limits_{c\in C(D^{+-})}u(c)x^{g_{c,D^{+-}}(v)}
			+\sum\limits_{c\in C(D^{--})}u(c)x^{g_{c,D^{--}}(v)}\\
			=&u(c_1)(x^{g_{c_1^+,D^{++}}(v)}-x^{g_{c_1^-,D^{-+}}(v)}-x^{g_{c_1^+,D^{+-}}(v)}+x^{g_{c_1^-,D^{--}}(v)})+\\
			&u(c_2)(x^{g_{c_2^+,D^{++}}(v)}-x^{g_{c_2^+,D^{-+}}(v)}-x^{g_{c_2^-,D^{+-}}(v)}+x^{g_{c_2^-,D^{--}}(v)}).
		\end{aligned}
		$
	\end{center}

	\begin{figure}[ht]
		\centering
		\includegraphics[width=0.88\textwidth]{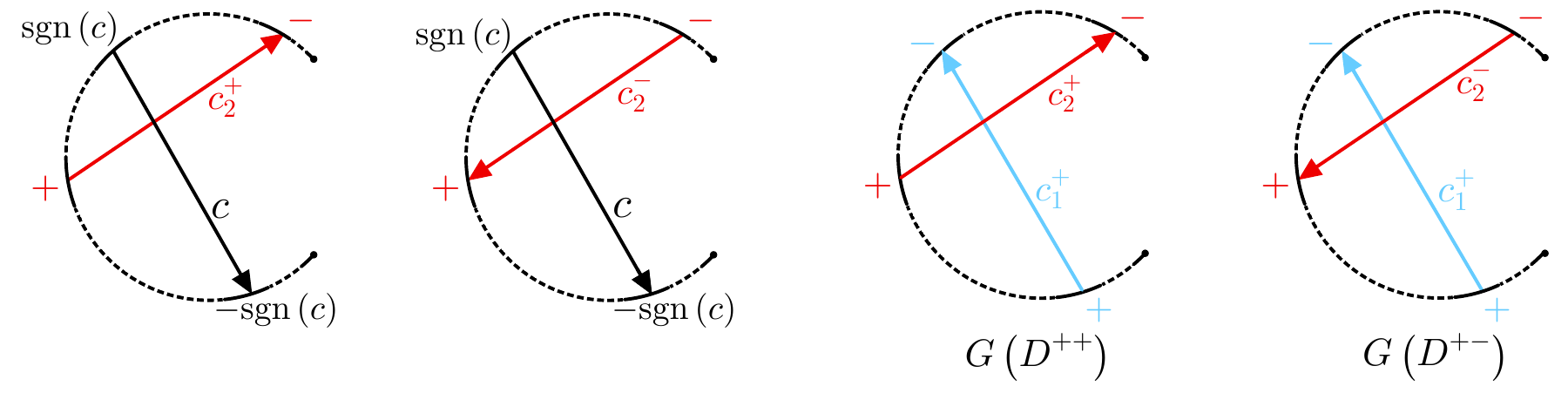}
		\caption{Gauss diagrams of knotoid diagrams relevant to $D(c_1,c_2)$}
		\label{21}
	\end{figure}

	The last step follows from the fact that $g_{c,D^{++}}(v)=g_{c,D^{-+}}(v)=g_{c,D^{+-}}(v)=g_{c,D^{--}}(v)$ for each crossing $c$ of $D(c_1,c_2)$ apart from $c_1$ and $c_2$, which can be easily verified by Figure \ref{21}. $\forall c\in C(D)\setminus \left\lbrace c_1,c_2\right\rbrace $, it contributes equally to $g_{c_1^+,D^{++}}(v)$ and $g_{c_1^+,D^{+-}}(v)$, so what really matters is the contribution of $c_2$. Apparently, $\sgn(c_2^+)=1$ and $\sgn(c_2^-)=-1$. Moreover, we can find that $d(c_2^+)=-d(c_2^-)$ and the direction of $c_2^+$ in $G(D^{++})$ and the direction of $c_2^-$ in $G(D^{+-})$ are opposite in Figure \ref{21}. As a consequence,
	\begin{center}
		$
		\begin{aligned}
			&g_{c_1^+,D^{++}}(v)-g_{c_1^+,D^{+-}}(v)\\
			=&(\varepsilon \sgn(c_2^+)v^{\phi_{c_1}(\varepsilon d(c_2^+))})-(-\varepsilon \sgn(c_2^-)v^{\phi_{c_1}(-\varepsilon d(c_2^-))})\\
			=&\varepsilon v^{\phi_{c_1}(\varepsilon d(c_2^+))}-\varepsilon v^{\phi_{c_1}(-\varepsilon d(c_2^-))}\\
			=&\varepsilon v^{\phi_{c_1}(\varepsilon d(c_2^+))}-\varepsilon v^{\phi_{c_1}(\varepsilon d(c_2^+))}=0,
		\end{aligned}
		$
	\end{center}
	which gives  $g_{c_1^+,D^{++}}(v)=g_{c_1^+,D^{+-}}(v)$, where $\varepsilon\in \left\lbrace -1,1\right\rbrace $ depends on the direction of $c_1^+$ and $c_2^+$. A similar argument yields $g_{c_1^-,D^{-+}}(v)=g_{c_1^-,D^{--}}(v)$, $g_{c_2^+,D^{++}}(v)=g_{c_2^+,D^{+-}}(v)$ and $g_{c_2^-,D^{-+}}(v)=g_{c_2^-,D^{--}}(v)$. Hence $B_{D(c_1,c_2)}=0$.

	Let $D_0(c)$ be the knotoid diagram with exactly 1 singular point $c$ shown in Figure \ref{22}, then by Table \ref{Table 3}, $B_{D_0(c^-)}(x,v)=x^{-1}-x$. We can see that $D_0(c^+)$ is isotopic to a trivial knotoid diagram after an $\Omega_2$ move, therefore 
	$B_{D_0(c^+)}(x,v)=0$. As a result, 
	\begin{center}
		$B_{D_0}(x,v)=B_{D_0(c^+)}(x,v)-B_{D_0(c^-)}(x,v)=x-x^{-1}\neq 0$
	\end{center} 
	when $x\neq \pm 1$, which completes the proof.

	 \end{proof}

	\begin{figure}[ht]
		\centering
		\includegraphics[width=0.6\textwidth]{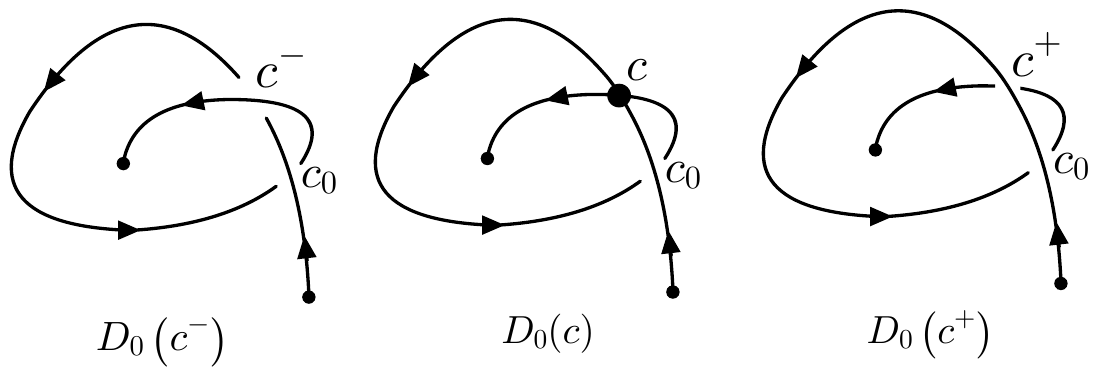}
		\caption{$D_0(c)$, $D_0(c^-)$ and $D_0(c^+)$}
		\label{22}
	\end{figure}

	\begin{table}[h!]
		\centering
		\renewcommand{\arraystretch}{1.4}
		\begin{tabular}{|c|c|c|}
			\hline
			$D_0(c^-)$&$c_0$&$c^-$\\
			\hline
			$u(c)$&1&$-1$\\
			\hline
			$d(c)$&$-1$&1\\
			\hline
			$g_c(v)$&$-1$&$1$\\
			\hline
		\end{tabular}
			\caption{$D_0(c^-)$}
			\label{Table 3}
	\end{table}

	\begin{proposition}\label{Proposition 4.2.} Let $D_1D_2$ be the product of $D_1$ and $D_2$, then\\ (1)$B_{D_1D_2}(x,v)=B_{D_1}(x,v)+B_{D_2}(x,v)$,\\ (2)$C_{D_1D_2}(x,v)=C_{D_1}(x,v)+C_{D_2}(x,v)$,\\
		(3)$\widetilde{B}_{D_1D_2}(x,v)=\widetilde{B}_{D_1}(x,v)+\widetilde{B}_{D_2}(x,v)$,\\ (4)$\widetilde{C}_{D_1D_2}(x,v)=\widetilde{C}_{D_1}(x,v)+\widetilde{C}_{D_2}(x,v)$.\\
	\end{proposition}

	\begin{proof} Here we only prove the first equality as an instance. It is clear that $u_{D_1D_2}(c)=u_{D_i}(c)$ for each crossing $c$ in $D_i$, $i=1,2$. $\forall c_1\in C(D_1), c_2\in C(D_2)$, the endpoints of $c_2$ lie on one side of the chord $c_1$ in $G(D_1D_2)$, hence $c_2$ contributes 0 to $d_{D_1D_2}(c_1)$, which gives $d_{D_1D_2}(c_1)=d_{D_1}(c_1)$. Moreover, $c_1\cap c_2=\varnothing$ for the chords $c_1$ and $c_2$ in $G(D_1D_2)$, thus $c_2$ contributes 0 to $g_{c_1,D_1D_2}(v)$ as well, which implies $g_{c_1,D_1D_2}(v)=g_{c_1,D_1}(v)$. Analogously, $g_{c_2,D_1D_2}(v)=g_{c_2,D_2}(v)$. In a word, $g_{c,D_1D_2}(v)=g_{c,D_i}(v)$ for each crossing $c$ in $D_i$, $i=1,2$. Finally we obtain $B_{D_1D_2}(x,v)=B_{D_1}(x,v)+B_{D_2}(x,v)$.
	
	\end{proof}

	In Example \ref{Example 4.1.}, we have seen that $C_D(x,v)=0$. Indeed, among the 4-phases functions, $C_D(x,v)$ is more likely to be trivial in some cases. The corollary given below indicates that it is not a coincidence.

	Two planar knotoid diagrams are said to be homotopic if one can be deformed into the other by a finite sequence of crossing changes.

	\begin{proposition}\label{Proposition 4.1.} Let $D_1$ and $D_2$ be two homotopic planar knotoid diagrams. Then $C_{D_1}(x,v)=
		C_{D_2}(x,v)$.
	\end{proposition}

	\begin{proof} 
		By the hypothesis, $D_2$ can be obtained from $D_1$ by a finite sequence of crossing changes, hence we may assume that $D_1$ and $D_2$ are two planar knotoid diagrams which differ by exactly one crossing change at a crossing $c_0$, and the conclusion follows by an induction on the original number of crossing changes.

		The Gauss diagrams corresponding to $D_1$ and $D_2$ are shown in Figure \ref{17}. By the figure it is easy to see that $\rho_{D_1}(c_0)=\rho_{D_2}(c_0)$, $\lambda_{D_1}(c_0)=\lambda_{D_2}(c_0)$ and $d_{D_1}(c)=d_{D_2}(c)$, which indicates that $c$ contributes equally to $h_{c_0,D_1}(v)$ and $h_{c_0,D_2}(v)$ for an arbitrary crossing $c$ with $c_0\cap c\neq \varnothing$ for the corresponding chords in the Gauss diagrams, hence $h_{c_0,D_1}(v)=h_{c_0,D_2}(v)$. Moreover, for such a crossing $c$, since $\sgn_{D_2}(c_0)=-\sgn_{D_1}(c_0)$, $d_{D_2}(c_0)=-d_{D_1}(c_0)$ and the fact that the directions of the chord $c_0$ in $G(D_1)$ and $G(D_2)$
		are opposite, we have
		\begin{center}
			$
			\begin{aligned}
				h_{c,D_1}(v)-h_{c,D_2}(v)&=\varepsilon \sgn_{D_1}(c_0)v^{\phi_c(\varepsilon d_{D_1}(c_0))}-\left(-\varepsilon \sgn_{D_2}(c_0)v^{\phi_c(-\varepsilon d_{D_2}(c_0))}\right)  \\
				&=\varepsilon \sgn_{D_1}(c_0)v^{\phi_c(\varepsilon d_{D_1}(c_0))}-\varepsilon\sgn_{D_1}(c_0)v^{\phi_c(\varepsilon d_{D_1}(c_0))}=0,
			\end{aligned}
			$
		\end{center}
		where $\varepsilon\in\left\lbrace -1,1\right\rbrace $ depends on the direction of $c_0$. By the illustration above, obviously $C_{D_1}(x,v)=C_{D_2}(x,v)$.

	\end{proof}

	\begin{figure}[ht]
		\centering
		\includegraphics[width=0.48\textwidth]{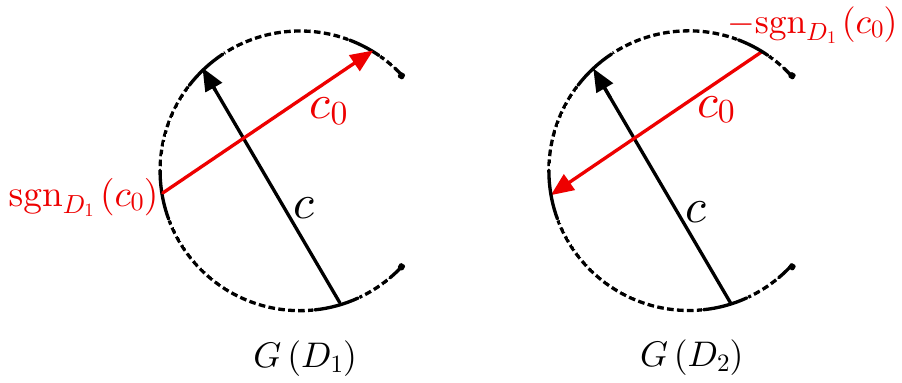}
		\caption{$G(D_1)$ and $G(D_2)$}
		\label{17}
	\end{figure}

	\begin{corollary}\label{Corollary 4.3.} Let $D$ be a planar knotoid diagram. If $D$ is homotopic to a trivial knotoid diagram, then $C_D(x,v)=0$.
	\end{corollary}

	Corollary \ref{Corollary 4.3.} follows from the fact that there is no crossing in a trivial knotoid diagram. In this way we can see that $C_D(x,v)$ vanishes at least for those knotoid diagrams which are homotopic to a trivial knotoid diagram. Therefore it is natural to ask if this is true for all planar knotoid diagrams, hence the question given below arises naturally:

	\noindent\textbf{Question.} Is it true for every planar knotoid diagram $D$ that $C_D(x,v)=0$?

	However, this question has not been solved yet, since we are not familiar with the homotopy classes of planar knotoid diagrams which have not been studied thoroughly until now, and it is quite tough to decide whether $C_D(x,v)$ vanishes or not for an arbitrary planar knotoid diagram which is not homotopic to a trivial knotoid diagram without a clear picture on the homotopy classes of planar knotoid diagrams. We are looking forward to an answer to this question.

\section*{Acknowledgements}
This work was supported by a grant of NSFC (No. 12331003) and the Fundamental Research Funds for the Central Universities (No. DUT25LAB302).

	\bibliographystyle{plain}
	\bibliography{4p-references.bib}

\end{document}